\documentclass[11pt]{amsart}
\usepackage[margin=1.2in]{geometry}
\usepackage{amsmath}
\usepackage{comment}
\usepackage{soul}
\usepackage{amsthm}
\usepackage{color}
\usepackage{lmodern}
\usepackage{amsfonts}
\usepackage{amssymb}
\usepackage{cancel}
\usepackage{amscd}%%% 
\numberwithin{equation}{section}
\newtheorem{theorem}{Theorem}[section]

\newtheorem{conj}[theorem]{Conjecture}
\newtheorem{remark}[theorem]{Remark}
\newtheorem{proposition}[theorem]{Proposition}
\newtheorem{prop}[theorem]{Proposition}
\newtheorem{lemma}[theorem]{Lemma}
\theoremstyle{definition}
\newtheorem{defi}{Definition}[section]
\newtheorem{rem}{Remark}[section]

\newcommand\half{\frac{1}{2}}

\DeclareMathOperator{\ad}{ad}

\newcommand\be{\beta}

\newcommand\g{\mathfrak g}

\newcommand\D{\Delta}
\renewcommand\l{\lambda}

\renewcommand\d{\delta}

\renewcommand\a{\alpha}
\renewcommand\aa{\mathfrak a}

\newcommand{\Z}{\mathbb Z}

\newcommand\s{\sigma}

\renewcommand\aa{\mathfrak a}

\newcommand\C{\mathbb C}
\newcommand\R{\mathbb R}

\newcommand{\fg}{\mathfrak{g}}

\newcommand{\ZZ}{\mathbb{Z}}

\newcommand{\vac}{{\bf 1}}
\newcommand{\bea}{\begin{eqnarray}}
\newcommand{\eea}{\end{eqnarray}}
\begin{document}

\title[ Conformal embeddings for $W$--algebras]{ New approaches  for studying  conformal embeddings and collapsing levels for $W$--algebras}
%: hook and rectangular type $\mathcal W$--algebra cases }

\author[Adamovi\'c, M\"oseneder, Papi]{Dra{\v z}en~Adamovi\'c}
\author[]{Pierluigi M\"oseneder Frajria}
\author[]{Paolo  Papi}
 \begin{abstract}
 
 %In this paper we extend our methods of detecting conformal embeddings in affine $\mathcal W$--algebras. 
 In this paper  we prove a general result saying that  under  certain hypothesis an embedding of an affine vertex algebra into an  affine $W$--algebra is conformal if and only if their central charges coincide. This result extends our previous result obtained in the case of minimal affine $W$-algebras  \cite{AKMPP-JA}. We also find a sufficient condition showing that certain conformal levels are collapsing.    
 This new condition enables us  to  find some levels $k$ where  $W_k(sl(N), x, f )$  collapses to its affine part when  $f$ is of   hook or rectangular type. Our methods can be applied to non-admissible levels.
  In particular, we prove  Creutzig's conjecture \cite{TC} on  the conformal embedding  in  the hook type $W$-algebra $W_k(sl(n+m), x, f_{m,n})$ of its affine vertex subalgebra.
  
  Quite surprisingly, the  problem of showing that certain conformal levels are not collapsing turns out to be very difficult.   In the cases when $k$ is admissible and conformal,  we prove that $W_k(sl(n+m), x, f_{m,n})$ is not collapsing.
  Then, by generalizing the results on semi-simplicity of conformal embeddings from \cite{AKMPP}, \cite{AKMPP-JJM}, we  find many cases in which
$W_k(sl(n+m), x, f_{m,n})$ is semi-simple as a module for its affine subalgebra at conformal level and we provide   explicit decompositions.
   
  \end{abstract}
\date{\today}

\maketitle
\section{Introduction}
An embedding $i: U\to V$ of a vertex operator algebra $(U,\omega')$ into a vertex operator algebra $(V,\omega)$ is called conformal if $i(\omega')=\omega$.
This  definition is a natural generalization of a notion  which was popular in physics literature in the mid 1980s, due to its relevance for string compactifications.
%, which involved comparing the Sugawara central charges of the affinizations of a pair $(\g^0,\g)$, consisting of a
%semisimple finite dimensional Lie algebra $\g$ and a reductive quadratic subalgebra, when acting on a level 1 integrable modules. 
In a series of  papers  \cite{AKMPP,AKMPP-Selecta, AKMPP-JJM, AMPP-Advances} we studied  conformal embeddings associated to affine vertex algebras $V^{k'}(\g'),V^k(\g)$ where $\g$ is a  basic Lie superalgebra (cf. \S \ref{1.1}) and $\g'$ ranges over a suitable class of Lie subsuperalgebras.
In \cite{AKMPP-JA} we initiated the study of conformal embeddings of affine vertex algebras in minimal affine $W$--algebras. \par  In this paper we study  embeddings into more general affine $W$--algebras which are not of minimal type, as an application of an abstract criterion working for conformal vertex algebras (cf. Definition \ref{svoa}) with assumptions on strong generators of  weight $2$.\par  Affine $W$--algebras can be regarded as a far reaching generalization of the superconformal algebras arising in Conformal Field Theory.   They are the  vertex algebras
$W^k(\g,x,f),$ constructed in \cite{KRW}, \cite{KW} by quantum Hamiltonian reduction starting from  a datum $(\g,x,f)$ and $k\in\C$ (see \S 1.1). Two relevant features of the vertex algebras $W^k(\g,x,f)$ are the following: 
\begin{itemize}
\item $W^k(\g,x,f)$  contains an affine vertex algebra  $V^{\beta_k}(\g^\natural)$, where   $\g^\natural$ is the centralizer in $\g$ of an $sl_2$-triple containing $f$  and  $\beta_k$ is the bilinear  form on $\g^\natural$ defined in \cite[Theorem 2.1 (c)]{KW}  (cf. \S \ref{ava} for notation). 
\item $W^k(\g,x,f)$ admits a unique simple graded quotient $W_k(\g,x,f),$ when $k\ne -h^\vee$.
\end{itemize} 
%Moreover, when $f$ is a {\it minimal} nilpotent element, the OPE expansion of the product of strong generators for $W^k(\g,x,f)$  is known. \par 
An important and extensively studied instance  of this construction is the {\it minimal} case, i.e. $f=f_\theta$ is a root vector of a minimal even root $\theta$.
Motivated by our study of conformal embeddings,  we found in  \cite{AKMPP-JA}  that a minimal simple affine $W$-algebra at certain levels can collapse to its affine vertex algebra. This leads to introduce the following definition.
Denote by $\mathcal V(\g^\natural)$ the image of $V^{\beta_k}(\g^\natural)$ under the projection $W^k(\g,x,f)\to W_k(\g,x,f)$.   \begin{defi}\cite{AKMPP-JA} We say  that a level $k$ is  {\sl collapsing} if $\mathcal V(\g^\natural)= W_k(\g,x,f)$.\end{defi}
A complete classification of collapsing levels in the minimal case  was presented in \cite{AKMPP-JA}.   
Applications of collapsing levels  to  semi-simplicity of the category $KL_k$ for affine vertex algebras were presented in \cite{AKMPP-IMRN}, \cite{AMP-super}, \cite{APV21}. Moreover, at collapsing levels, one has the structure of a  vertex  tensor category \cite{CY}.

In the recent paper \cite{AEM}  T. Arakawa, J. van Ekeren, and A. Moreau   started  the study of collapsing levels in general affine $W$-algebras.  They focus on  the cases when $k$ and the level $k^\natural$ of $\mathcal V(\g^\natural)$ are both admissible. Admissibility implies the modularity of characters, which enables them to define the   asymptotic datum of a  representation. Then, if $f$ belongs to the associated variety of 
$V^k(\g)$ (a certain affine Poisson variety),  they provide a criterion for $k$  to be collapsing in terms of asymptotic data.  Although their method is not applicable if $k^{\natural}$ is not admissible, some of their results provide conjectures on the existence of collapsing levels beyond the admissible range. Moreover, 
in a recent physics paper on $4$d SCFT,  B. Li, D.  Xie  and W. Yan \cite{LXY}  study various quantities (central charges, flavor symmetry, Higgs
branch, etc) which  suggest the existence  of new collapsing levels. 

  In  the present paper we consider  conformal embeddings and collapsing levels associated to affine $W$--algebras for general, not necessarily admissible levels. 
  We are able to prove some conjectures of from \cite{AEM} and \cite{LXY} on collapsing levels. We focus on the following specific  topics.
\begin{itemize}
\item A criterion for conformal embeddings and collapsing levels.
\item Hook type $W$--algebras and rectangular $W$--algebras for $sl(N)$.
\item Conformal vs collapsing levels.
\item Decomposition of conformal embeddings: hook $W$--algebra case.
%\item Future work.
\end{itemize}
We now expand on each topic.
\subsection{A criterion for conformal embeddings and collapsing levels}\label{1.1} Let $\g$ be a basic Lie superalgebra, i.e.  a
simple finite--dimensional Lie superalgebra with a reductive even part and a non-zero even
invariant supersymmetric bilinear form $(. \, | \, .)$. Recall \cite{KW} that the universal affine
$W$-algebra   $W^k(\g,x,f)$ of central charge $c(\g, k)$ is  associated to the datum  $(\fg ,x,f)$, where  $\fg$ is a
basic Lie superalgebra, $x$ is an
$\ad$--diagonalizable element of $\fg$ with eigenvalues in
$\tfrac{1}{2}\ZZ$, $f$ is an even  nilpotent element of $\fg$ such that
$[x,f]=-f$ and the eigenvalues of $\ad x$ on the centralizer
$\fg^f$ of $f$ in $\fg$ are non-positive. We will also assume that the datum $(\g,x,f)$ is {\it Dynkin}, i.e. there is a $sl(2)$-triple $\{e,h,f\}$ and $x=\half h$. Let $\g^\natural$ be the centralizer of this  $sl(2)$-triple. 
%It turns out that $W^k(\g,x,f)$ contains   an affine vertex subalgebra $V(\g^\natural)$ and that $W^k(\g,x,f)$  has a a unique simple graded quotient 
% $W_k(\g,x,f)$  (at non critical level).
In \cite{AKMPP-JA} we proved in the minimal case  that  $\mathcal V(\g^{\natural})$  is  conformal if and only if 
\bea c_{{\g}^{\natural}}  = c(\g, k) = \frac{ k\, \mbox{sdim} \g}{k+ h^{\vee}} - 6 k + h^{\vee} -4,  \label{cc-min}\eea
where $c_{\g^{\natural}}  $ is the Sugawara central charge  of the affine vertex subalgebra $\mathcal V({\g}^{\natural}) \hookrightarrow W_k(\g, x,  f_{\theta})$ and $h^\vee$ is the dual Coxeter number of $\g$.
It is clear that (\ref{cc-min}) is a necessary condition  for a conformal embedding. But proving that this condition is sufficient was  highly non-trivial and it was a central part of  \cite{AKMPP-JA}. 
Recall from \cite{KW} that a minimal $W$-algebra $W_k(\g, x, f_{\theta})$  is strongly generated  by a Virasoro field $L$ and fields  $J^{\{u\}}, G^{\{v\}}$ of conformal weight $1,3/2$, respectively.
We proved that

\begin{enumerate}
\item a sufficient condition for the  embedding $\mathcal V({\g}^{\natural}) \hookrightarrow W_k(\g, x,  f_{\theta})$ to be conformal is that the  {\it Sugawara} conformal weight of the  fields $G^{\{v\}}$ is $\frac{3}{2}$; 
\item the previous condition implies that the conformal level $k $ is either $- \frac{2h ^{\vee} }{3}$ or $- \frac{h^{\vee}-1}{2}$. 
\end{enumerate}
 Statements (1) and (2) easily imply that (\ref{cc-min}) is   a sufficient condition for conformal embedding.

In the present paper we provide a generalization of criterion  (\ref{cc-min})  applying also to  non-minimal affine $W$--algebras.  Roughly speaking, we show that if one has enough knowledge about the weight two subspace of strong generators then one can check that the  embedding is  conformal by comparing central charges.  More precisely, we prove the following general result for a conformal vertex algebra $
W=\bigoplus\limits_{n\in{ \frac{1}{2}\mathbb Z_+}}W(n),
$ of central charge $c$ satisfying     the technical hypothesis  \eqref{h1}, \eqref{h2}, \eqref{h3} below. Let $L$ be the conformal vector of $W$ and $L^\aa$ that of the vertex algebra $V(\aa)$ attached to the Lie superalgebra $\aa=W(1)$ with $\l$-bracket \eqref{aff}. Note that $V(\aa)$ is an affine vertex algebra in the sense of \S\ref{ava}. In Theorem \ref{Criterion} we prove
\begin{theorem}\label{Criterion-intr} 
Assume that $W$ is strongly generated by $\mathfrak a=W(1)$ and by  $\{L-L^\aa\}\cup S$ with $S$ homogeneous  and such that
 $(L-L^{\mathfrak a})(2)X=0$ for $X\in S\cap W(2)$.
Then $(L-L^{\mathfrak a})$ generates a  proper ideal $I$ in $W$ if and only if $c=c_\aa$.  In other words,  $W/I$ is non-zero and the image of $V(\aa)$ in $W/I$  is conformally embedded  if and only if $c=c_\aa$.
\end{theorem}
\noindent Note that the weight two subspace is an $\aa$-module, and a sufficient condition for our criterion to hold is that the trivial representation of $\aa$ appears with multiplicity one.

Next we prove a sufficient condition for  a conformal level $k$ to be collapsing.
Assume    that $c=c_\aa$ and let $I$ be, as above,  the proper ideal generated by $L-L^\aa$. Let $\pi_I: W\to W/I=\overline W$ be the quotient map.
{ Set $V=span(S)$. Choose the  strong generators  in such a way that  $V$ is both $\aa$-stable and $L(0)$-stable. Decompose
$$V=\bigoplus_{i\in\mathcal J} V_i$$
into irreducible  $\aa$-modules. Since $L(0)$ and $\aa$ commute, we can assume that $V_i$ is homogeneous for $L(0)$ of conformal weight $\D_i$. Let also $C_{i}$ be the eigenvalue of $L^\aa$ on the  highest weight vector of the $V(\aa)$-module with top component
$V_i$ (cf. \eqref{Cj}).
Then $\overline W$ is strongly generated by $\aa \oplus \sum_{i\in\mathcal K}\pi_I(V_i), \mathcal K \subset \mathcal J$.
 We assume $\mathcal K$ to be minimal in the sense that, if $\mathcal T$ is  a proper subset of $\mathcal K$, then $\aa\oplus (\sum_{i\in \mathcal T}\pi_I(V_i))$ does not strongly generate $\overline W$. In Theorem \ref{Criterion2} we prove
 
\begin{theorem}\label{Criterion2-intr} With the above assumptions, $C_j= \D_j$ for all $j\in\mathcal K$. In particular, if $C_i\ne \D_i$ for all $i\in \mathcal J$, then $\mathcal K=\emptyset$, hence
$\overline W$  collapses to $V(\mathfrak a)  / (I \cap V(\mathfrak a))$.
\end{theorem}

\subsection{Hook type $W$--algebras and rectangular $W$--algebras}
 Hook type  $W$--algebras  recently appeared in \cite{LS}  as coset vertex algebras,  and also  in the context of dualities and trialities of various vertex algebras \cite{CL, CL2,CLNS}. In these papers  the authors mainly stu\-died the generic level case. Since  conformal levels are not generic, we derive  in  Theorem \ref{41}  the  structure of these $W$-algebras at arbitrary level, using only very general results by  Kac and Wakimoto \cite{KW}.   A similar approach is applied in Theorem \ref{genrect} for the rectangular $W$--algebras.

We  first apply Theorems \ref{Criterion-intr} and \ref{Criterion2-intr} to the hook type $W$-algebra $W_k(\g, x, f_{m,n})$ for $\g=sl(m+n)$ (i.e., the partition representing the nilpotent element $f_{m,n}$ is the hook $(m,1^n)$). The outcome is the following result (see Theorems \ref{MT}, \ref{coll}).

\begin{theorem}\label{MT-intr}
 
\item[(1)]The embedding $\mathcal V(\g^{\natural}) \hookrightarrow W_k(\g, x,f_{m,n})$ is conformal if and only if $$k=k_{m,n} ^{(i)},\quad 1\le i\le 4,$$ where
 $k_{m,n} ^{(1)}  =-\frac{m}{m+1} h^{\vee}$ ($n >1$),
 $k_{m,n}^{(2)} =   
-\frac{(m-1) h^{\vee} -1}{m}$ ($n \ge 1$),
 $ k_{m,n}^{(3)} =  - \frac{(m-2) h^{\vee}+1}{m-1} $  ($n \ge 1$, $m >1$),
 $k_{m,n}^{(4)} = -\frac{(m-1)h^{\vee}}{m} $.
 \item[(2)] Levels $k_{m,n}^{(3)}$ ($m \ne n-1$) and $k_{m,n}^{(4)}$ are collapsing.
 \end{theorem}
\noindent The case $k=k_{m,n} ^{(1)}$ solves a conjecture of Creutzig  \cite{TC}.
%(see Remark \ref{Creu}).\par
Next we consider the rectangular case and the $W$-algebra $W_k(\g, x, f)$ for $\g=sl(m q)$ (i.e., the partition representing the nilpotent element $f$ is rectangular of shape  $(q^m)$). In this case we obtain the following theorem (see Theorem \ref{collapsing-rectangular}).

\begin{theorem}\label{rectangular-intr}
 
\item[(1)]The embedding $\mathcal V(\g^{\natural}) \hookrightarrow W_k(\g, x,f_{m,n})$ is conformal if and only if $$k=k_{m,q} ^{[i]},\quad 1\le i\le 3,$$ where
 $k_{m,q} ^{[1]}  =-\frac{m q^2}{q+1}  $, 
 $k_{m,q}^{[2]} = -\frac{-m q^2 + m q-1}{q}$, 
 $ k_{m,q}^{[3]} =  -\frac{-m q^2 + m q+1}{q}$.
   \item[(2)] Levels $k_{m,n}^{[i]}$, $i=1,2,3$  are collapsing for all $q \ge 2$ and $m \ge 3$  and we have
     \begin{align}
W_{k^{[1]} _{m,q} }(\g,x,f_{m,n})&=V_{-\frac{m q}{q+1}}(sl(m)),\nonumber\\  
W_{k^{[2]} _{m,q} }(\g,x,f_{m,n})&=V_{-1}(sl(m)) , \nonumber \\ 
W_{k^{[3]} _{m,q} }(\g,x,f_{m,n})&=V_{1}(sl(m)).  \nonumber
\end{align}
\end{theorem}

 We should mention that in the cases when $k_{m,q}^{[i]}$ is admissible, the results of the previous theorem are (or can be) obtained by using methods from \cite{AEM}. Collapsing of the  non-admissible $W$--algebra $W_{k^{[2]} _{m,q} }(sl(mq),x,f_{m,n})$ was conjectured in \cite{AEM} (see Remark \ref{AAAA} below).

 \subsection{Conformal vs collapsing levels}

 Assume that $k$ is a conformal level of an affine $W$-algebra $W^k(\g, x, f)$. Then $\bar L = L - L^{\g^{\natural}}$ is a singular vector in $W^k(\g, x, f)$ and we can investigate  the ideal
 $I= W^k(\g, x, f). \bar L$ and the quotient vertex algebra
 $\overline{W}_k(\g, x, f) = W^k(\g, x, f) / {I}$. Since it can happen that $\overline{W}_k(\g, x, f) \ne  W_k(\g, x, f)$ (cf. Remark \ref{ex-nstr}), we introduce the following notion.
 
 \begin{defi}Let $\mathcal V(\g^\natural)$ be the image of $V^{\beta_k}(\g^\natural)$ in $\overline W_k(\g, x, f)$. 
 A conformal level $k$ is called {\it strongly collapsing} if $\mathcal V(\g^\natural) =  \overline W_k(\g, x, f)$.
 \end{defi}
 The strongly collapsing levels are easy to detect using our Theorem \ref{Criterion2-intr}.   We do have examples of levels which are collapsing but not strongly collapsing (see Remark \ref{ex-nstr}). These situations are difficult to detect in general, due to the lack of knowledge of explicit OPE formulas in non-minimal cases.
Clearly, every strongly collapsing level is also  collapsing. Levels described in Theorem \ref{Criterion2-intr}  are  strongly collapsing.
So when we investigate whether a certain level is collapsing, we first check  if such a level is strongly collapsing. 
%The next question is: can we detect a generator of $W^k(\g, x, f)$ which does not belong to $I$ ?
% We introduce a new method for solving this problem. It uses the fusion rules of the Virasoro vertex algebra ${\rm Vir}^{c=0}$  at central charge $0$ generated by the field $\bar L$. The procedure  is as follows:

%\begin{itemize}
%\item we check that Theorem \ref{Criterion2-intr} cannot be applied, so there is a generator $G$ such that $\bar L(0) G = 0$;
%\item motivated by the applications, we assume that $G$ is a singular vector, so $\bar M(0,0) = {\rm Vir}^{c=0}. G$ is a quotient of the  Verma module $M(0,0)$;
%\item by using fusion rules in the category of  ${\rm Vir}^{c=0}$--modules  (cf. Propositions \ref{prop-fusion-trick} and \ref{prop-fusion-trick-2}) we get that $G$ cannot appear in fusion products of other generators of the $W$-algebra.
%\end{itemize}
 
In the case of  hook $W$--algebras we get:
 \begin{prop}
 Levels $k_{m,n}^{(1)}$ and $k_{m,n} ^{(2)}$ are not strongly collapsing for  $W_k(sl(n+m), x, f_{m,n})$. \end{prop}
 But it still may happen that some of these levels are collapsing. 
 We conjecture:
 \begin{conj} \label{non-coll-1}$k= k^{(1)}_{m,n}$ is always non-collapsing. \end{conj}

 Using different techniques, it was proved in \cite{ACGY} that the level $k_{p-1,2}^{(1)}$ is not collapsing. But we prove  in Theorem \ref{collapsing-3p2} that $k_{3p, 2}^{(2)}$ is collapsing. The main difference between these two cases is that $k_{p-1,2}^{(1)}$ is admissible, whereas $k_{3p, 2}^{(2)}$ is not.
 % admissible for $sl(3p+2)$.
 
 The peculiarity
 of the admissible case is that the maximal ideal in $V^k(sl(n+m))$ is generated by one singular vector.  As a consequence, we prove in Theorem \ref{non-collapsing-admissible}  that  $\bar L$ generates the maximal ideal in $W^k(sl(n+m), x,f_{m,n})$.
Thus  if  $k= k^{(i)}_{m,n}$, $i=1,2$ is admissible, then $k$ is not collapsing.   Therefore  Conjecture \ref{non-coll-1} holds for admissible levels. We also present in Remark \ref{rem-non-coll-1} some arguments explaining why the conjecture should hold in general.

  \subsection{Decomposition of conformal embeddings: hook $W$--algebra case}  If the embedding $\mathcal V({\g}^{\natural}) \hookrightarrow W_k(\g, x, f)$ is conformal, it is natural to ask for the decomposition  of $W_k(\g, x, f)$ as a  $\mathcal V({\g}^{\natural})$-module.
 In general, describing  this decomposition is a  hard  problem, open in most cases.  Some decompositions are provided in our previous papers \cite{AKMPP,  AKMPP-Selecta, AKMPP-JJM,   AMPP-Advances}. In particular, the paper \cite{AKMPP-JJM} shows decompositions of minimal affine $W$--algebras $W_k(\g,x,f_{\theta})$. It turns out that we can generalize some results of   \cite{AKMPP-JJM} to hook type $W$--algebras.

\begin{theorem} \label{decomp-intr}   (cf. Theorem \ref{decomp})  
% ( {\color{blue} Assuming that  Lemma \ref{non-zero-2} holds}).
 Let $k=k_{m,n}^{(i)}$ for  $i=1,2$  and assume also that    $k$ is non-collapsing.   Assume that $\frac{m+1}{n-1} \notin {\Z}$  if  $i=1$ and $ \frac{m}{n+1}  \notin {\Z}$ if $i=2$. Then
 
\bea\label{deco} W_k(\g, x,  f_{m,n}) = \bigoplus_{i \in {\Z}}  W_k ^{(i)}, \eea
where each $W_k ^{(i)} = \{ v \in  W_k \ \vert \ J(0) v = i v \}$ is an irreducible $V(gl(n))$--module (cf. \eqref{J} for the definition of $J$).  
In particular, 	decomposition  \eqref{deco}  holds if $k$ is admissible and $n >2$.
\end{theorem}

\subsection{Future work}
The main purpose of the present paper is to introduce new general methods for studying conformal embeddings and collapsing levels for affine $W$--algebras. As an illustration, we apply our methods for the hook  and rectangular type $W$-algebras of type $A$.
 It is worthwhile to note that, since the present paper appeared in preprint form, our results on collapsing levels have already been applied 
by physicists in the context of four dimensional $N = 2$ superconformal field theories \cite{XY}.\par

 The same methods can  also be applied for hook type $W$--algebras of types $B, C, D$ and for hook type $W$--superalgebras (see \cite[Table 1]{CLNS} for the list of such vertex algebras). We hope to study these cases in our forthcoming papers.

The next important problem is to study decompositions of conformal embeddings. As we see in present paper, the first difficult step is to prove that a  given  conformal level is not collapsing. We obtained results in this direction using the fact that a  certain  conformal level is admissible, but  we hope that our arguments can be extended to non-admissible cases. A natural approach would be to use   $\lambda$--bracket for $W$--algebras. This approach works for algebras of low rank (cf. \cite{AMP-21}, \cite{JF}), where one can use  computer calculations. 

It was proved in  \cite{ACGY} that  $W_{k_{p-1,2} ^{(1)}} (sl(p+1), f_{p-1,2}) =  \mathcal R^{(p)}$, where $\mathcal R^{(p)}$ is a certain logarithmic vertex algebra realized  in \cite{A-TG} as an extension of $L_{-2+1/p}(gl(2))$. We believe that recent explicit realizations, motivated by inverting of quantum Hamiltonian reduction (as in the case of  $L_k(sl(3))$ from \cite{ACG}), can be used to get explicit decompositions for  conformal embeddings in some non-collapsing cases.

 We believe that at each collapsing level, the category $KL_k(\g)$ (cf. \cite[Definition 2.1]{ AMP-super}) is  semi-simple. This is proved   in \cite{AKMPP-IMRN, AMPP-Advances}   
 in the case of collapsing levels for the minimal reduction. The same is true for some collapsing levels in non-minimal cases (cf. \cite{APV21, ACPV22}). In all these examples we also have that $KL_k(\g)$ is a rigid braided tensor category \cite{CY}.  The proof of rigidity uses the conformal embedding of $\mathcal V(\g)$ into affine vertex algebras or minimal affine $W$--algebras. Our paper provides some tools for  a natural generalization. Decompositions in Theorem \ref{decomp-intr}  suggest that the $V(gl(n))$--modules in  (\ref{deco}) are simple currents in a suitable tensor category. Unfortunately, our current methods are not sufficient to prove such statement. We hope to study this problem elsewhere.

  \vskip 5mm
  \centerline{{\bf Acknowledgements.}}
     \vskip10pt
The authors thank the referees for their careful reading of our paper and for suggestions which improved our exposition. \par
D.A.   is  partially supported   by the
QuantiXLie Centre of Excellence, a project cofinanced
by the Croatian Government and European Union
through the European Regional Development Fund - the
Competitiveness and Cohesion Operational Programme
(KK.01.1.1.01.0004).

\section{Setup}
 \subsection{Conformal vertex algebras} Recall that a vector superspace  is a $\mathbb Z/2\Z$--graded vector space
 $W=W_{\bar{0}}\oplus W_{\bar{1}}$. The elements in $W_{\bar{0}}$ 
(resp. $W_{\bar{1}}$) are called even (resp. odd). Set
$$p(v)=\begin{cases}0\in\Z&\text{ if $v\in W_{\bar{0}}$},\\ 1 \in\Z&\text{ if $v\in W_{\bar{1}}$.}
\end{cases}
$$
We will regard  $p(v)$ as an integer, not as a residue class. We will often use the notation
\begin{equation}\label{s}\s(u)=(-1)^{p(u)}u,\qquad p(u,v)=(-1)^{p(u)p(v)}.\end{equation}

Let $W$ be a  vertex algebra.  We let 
 \begin{align}\label{0a3}
& Y:W \to (\mbox{End}\,W)[[z,z^{-1}]] ,\\
& v\mapsto Y(v,z)=\sum_{n\in{\Z}}v_{n}z^{-n-1}\ \ \ \  (v_{n}\in
\mbox{End}\,W),\notag
\end{align}
denote the state--field correspondence. We denote by $\vac$ the vacuum vector in $W$ and by $T$ the translation operator.

\begin{defi}\label{svoa}  A  {\it conformal} vertex algebra is a vertex algebra  $W$  such that  there exists a distinguished vector $L\in
W_2$,  called a Virasoro vector, satisfying the following conditions:
\begin{align} \label{0a4}
& Y(L,z)=\sum_{n\in\Z}L(n)z^{-n-2},\ [L(m),L(n)]=(m-n)L(m+n)+\frac{1}{12}(m^3-m)\delta_{m+n,0}c\,I ,\\
& L(-1)=T,\\
& \text{$L(0)$ is diagonalizable and its eigenspace decomposition has the form}\end{align}
\begin{equation}\label{g2.1}
W=\bigoplus_{n\in{ \frac{1}{2}\mathbb Z_+}}W(n),
\end{equation}
where 
\begin{equation}\label{dimf}\dim W(n)< \infty \text{ for all $n$ and $W(0)=\C\vac$}.
\end{equation}
The number $c$ is called the {\it central charge}.
\end{defi}
\subsection{Affine vertex algebras}\label{ava}
 Let $\aa$ be a Lie superalgebra equipped with an  invariant supersymmetric bilinear form $B$. The universal affine vertex algebra $V^B(\aa)$ is  the universal enveloping vertex algebra of  the non--linear  Lie conformal superalgebra $R=(\C[T]\otimes\aa)$ with $\lambda$-bracket given by
$$
[a_\lambda b]=[a,b]+\lambda B(a,b),\ a,b\in\aa.
$$
In the following, we shall say that a vertex algebra $V$ is an affine vertex algebra if it is a quotient of some $V^B(\aa)$.
%, and we use for $V$ the notation $\mathcal V(\aa)$ is $B$ is clear from the context. 
 If $k\in \C$ and $\aa$ is simple,  we will write simply $V^k(\aa)$ for $V^{k(\cdot |\cdot)}(\aa)$, where $(\cdot |\cdot)$ is a fixed normalized bilinear form. We  will always assume that $k$ is non--critical, i.e. $k\ne - h^\vee$. With this assumption, it is known that $V^k(\g)$ has a unique simple quotient, denoted by   $V_k(\g)$  (see \cite[4.7 and Example 4.9b]{K2}). The vertex algebras $V^k(\g), V_k(\g)$ are VOAs with Virasoro vector $L_\g$ given by the Sugawara construction.

\subsection{Invariant Hermitian forms on conformal vertex algebras} 
 Let $W$ be a conformal vertex algebra. If $a\in W(\D_a)$, set 
\begin{align}
(-1)^{L(0)}a&=e^{\pi\sqrt{-}1\D_a}a,\quad\s^{1/2}(a)=e^{\frac{\pi}{2}\sqrt{-}1p(a)}a.
\end{align}

If  $\phi$ is a conjugate linear involution of $W$, set 
 \begin{equation}\label{gamma}g=((-1)^{L(0)} \s^{1/2})^{-1} \phi.\end{equation} 
 Recall from \cite{KMP} the following definition.
 \begin{defi}\label{invariant_form}
Let $\phi$ be a conjugate linear involution of a conformal  vertex algebra $W$ such that $\phi(L)=L$. Let $g$ be as in \eqref{gamma}. A Hermitian form $(\, \cdot \,\, , \, \cdot\, )$ on $W$ is said to be \emph{$\phi$--invariant} if, for all $a\in W$, 
\begin{equation}\label{i}
(v,Y(a, z)u)=(
Y(e^{zL(1)}  z^{-2L(0)}ga, z^{-1})v,u),\quad u,v\in W.
\end{equation}
\end{defi}
\begin{rem}\label{existsinv}
The existence conditions for a $\phi$--invariant Hermitian form on a  vertex operator algebra $W$ are discussed in \cite[Theorem 4.3]{KMP}. For a conformal vertex algebra 
these conditions reduce to 
\begin{equation}\label{excond}
L(1)W(1)=\{0\}.
\end{equation}
In such a case we normalize the form by requiring $(\vac,\vac)=1$.
\end{rem}

\begin{rem}\label{ker} The kernel of (any) $\phi$--invariant Hermitian form on a conformal vertex algebra  coincides with its maximal  ideal.
\end{rem}

Recall that $\mathfrak a=W(1)$ is a Lie superalgebra with bracket defined by
$$
[a,b]=a_{0}b.
$$
Let  $\langle\cdot,\cdot\rangle$ be the bilinear form on $\aa$ defined by 
$$
\langle a,b \rangle=(g(a),b).
$$
If \eqref{excond} holds, $\aa$ is made of primary elements. It follows that
$$
\langle a, b\rangle\vac=(g(a),b)\vac=(g(a)_{-1}\vac,b_{-1}\vac)\vac=(\vac,a_{1}b)\vac=a_{1}b,
$$
so that we can write
\begin{equation}\label{aff}
[a_\l b]=[a,b]+\l \langle a,b\rangle\vac.
\end{equation}
Note that the form $\langle\cdot,\cdot\rangle$ is a supersymmetric invariant form.
Indeed
$$
\langle b,a \rangle=(g(b),a)=\overline{(a,g(b))}=\overline{g(a)_{1}g(b)}.
$$
As shown in Lemma 3.1 of \cite{KMP}, one has
\begin{equation}\label{gg=g}
g(v_{n}u)=(-1)^{n+1}p(v,u)g(v)_{n}g(u),
\end{equation}
so 
$$\langle b,a \rangle=p(a,b)\overline{g(a_{1}b)}=p(a,b)a_{1}b=p(a,b)(g(a),b)=p(a,b)\langle a,b\rangle.
$$
Since 
$$\langle a,[b,c]\rangle=\langle a,b_{0}c\rangle=( g(a),b_{0}c)=( g(b)_{0}g(a),c),
$$
 using \eqref{gg=g} again, we find
$$\langle a,[b,c]\rangle=-p(b,a)(g(b_{0}a),c)=-p(b,a)\langle[b,a],c\rangle=\langle[a,b],c\rangle.
$$
Therefore, by \eqref{aff}, the vertex subalgebra generated by $\mathfrak a$ is an affine vertex algebra that,  in this section,  we denote by $V(\mathfrak a)$.
 
We assume that $\aa=\bigoplus_{i=0}^r \aa_i$ with $\aa_0$ an even abelian Lie algebra (possibly $\{0\}$) and $\aa_i$ simple Lie algebras or basic Lie superalgebras.
\begin{lemma}\label{ideala}
Let $\aa'$ be the kernel of $\langle\cdot,\cdot\rangle$. Then $(\aa')_{-1}\vac$ generates a proper ideal $I_{\aa'}$ of $W$.
\end{lemma}
\begin{proof}
Since $\phi$ is an automorphism of $W$, if $a,b\in\aa$, then
$$
\phi([a_\l b])=[\phi(a)_\l\phi(b)]
$$
so
$\phi(\langle a,b\rangle\vac)=\langle \phi(a),\phi(b)\rangle\vac$
hence
$
\langle\phi(a),\phi(b)\rangle=\overline{\langle a,b\rangle}$. In particular $\phi(\aa')=\aa'$. It follows that
$$
(\aa',\aa)=(g(\aa'),\aa)=\langle \aa',\aa\rangle=0.
$$
Since $\aa=W(1)$ and $\aa'\subset W(1)$, we have $(\aa',W)=0$, hence $\aa'$ is in the kernel of the form $(\cdot,\cdot)$. We conclude using Remark \ref{ker}.
\end{proof}
\subsection{Casimir elements and conformal vectors}  Let $\langle\cdot,\cdot\rangle_i$ be the nondegenerate invariant form on $\aa_i$ normalized as follows: if $\aa_i$ is an even simple Lie algebra then require $\langle\theta_i,\theta_i\rangle=2$ (where $\theta_i$ is a long root of $\aa_i$).  If $\aa_i$ is not even, then we let $\langle\cdot,\cdot\rangle_i$ be the form described explicitly in Table 6.1 of of \cite{Kw}. Set $\aa'_0=\aa'\cap\aa_0$.
On $\aa_0$ we choose a basis $\{a_i\}$ such that $\langle a_i,a_j\rangle =0$ for $i\ne j$ and 
$$\langle a_i,a_i\rangle =\begin{cases}0&1\le i\le \dim\aa'_0,\\
1& \dim\aa'_0<i\le \dim\aa_0.
\end{cases}
$$
Set $\aa''_0=span(a_i\mid i>\dim\aa'\cap\aa_0)$. Set $\langle\cdot ,\cdot\rangle_0=\langle\cdot,\cdot\rangle_{|\aa_0\times\aa_0}$.
%and let $\langle\cdot,\cdot\rangle_0$ to be the bilinear symmetric form such that $\langle a_i,a_j\rangle_0=\d_{ij}$.
Since $\langle\cdot, \cdot\rangle$ is invariant we have 
\begin{equation}\label{formdegenerate}
\langle\cdot,\cdot\rangle_{|\aa_i\times\aa_i}=k_i\langle\cdot,\cdot\rangle_i\ (i\ge1).
\end{equation}
If $\langle\cdot,\cdot\rangle_0=0$ we let $k_0=0$ and $k_0=1$ otherwise, so that \eqref{formdegenerate} holds also for $i=0$.

It follows that $V(\aa)$ is a quotient of $V^0(\aa'_0)\otimes V^1(\aa''_0) \otimes(\bigotimes_{i>0} V^{k_i}(\aa_i))$ 
%(where $V^{k}(\g)$ is the universal affine vertex algebra at level $k$ attached to $\g=\aa'_0,\aa''_0,\aa_i$).
We assume that $k_i\in\R$ for all $i$.

For $i>0$, let $C_{\aa_i}$ be the Casimir element of $\aa_i$ corresponding to $\langle\cdot,\cdot\rangle_i$ and let $2h_i^\vee$ be the eigenvalue of its action on $\aa_i$. For $i=0$ we let 
$$C_{\aa_0}=\begin{cases}0\quad&\text{if $\aa_0=\aa'_0$,}\\
\sum\limits_{i=\dim\aa'_0+1}^{\dim\aa_0}a_i^2 \quad&\text{otherwise,}
\end{cases}$$
 so that $h^\vee_0=0$.

We assume that, for $i>0$,  $k_i$ are non-critical, i.e. $k_i\ne - h^\vee_i$. This implies that $V^{k_i}(\aa_i)$ admits a Virasoro vector $L^{\aa_i}$ given by Sugawara construction. We set also 
$$L^{\aa_0}=\half \sum\limits_{i=\dim\aa'_0+1}^{\dim\aa_0}:a_ia_i:.
$$ 

This is a Virasoro vector for $V^1(\aa_0'')$.
Define $L^\aa\in W$ to be the image of
$\sum_{i\ge 0}L^{\aa_i}$ in $V(\aa)$. In general this is not a Virasoro vector for $V(\aa)$. Let $I_{\aa'_0}$ be the ideal generated by $\aa'_0$. By Lemma \ref{ideala}, the ideal $I_{\aa'_0}$ is proper. Let $\pi_{\aa'_0}:W\to W/I_{\aa'_0}$ be the quotient map.  Then $\pi_{\aa'_0}(L^\aa)$ is a Virasoro vector for  $\pi_{\aa'_0}(V(\aa))$.
Note that
$$
L^{\mathfrak a}(2)L^{\mathfrak a}=\tfrac{1}{2}c_{\mathfrak a}\vac.
$$
By a slight abuse of terminology, we call $c_\aa$ the {\it central charge} of $L^\aa$.

\section{A criterion for conformal embedding of affine vertex algebras into    conformal vertex algebras}

Let $W$ be  a conformal vertex algebra with conformal vector $L$. Make the following assumptions on $W$:
\begin{align}
\label{h1} &\text{there is a conjugate linear involution  $\phi$ of $W$ with $\phi(L)=L$;}\\
\label{h2} &W(\tfrac{1}{2}) = \{0\};\\
\label{h3} &\aa=W(1)\text{ consists of $L$--primary elements.} 
\end{align}
By Remark \ref{existsinv}, there is a $\phi$-invariant Hermitian form $(\cdot,\cdot)$ on $W$ such that $(\vac,\vac)=1$. \par 

 We say that a subset $\mathfrak V$ of a vertex algebra $W$, homogeneous with respect to parity, strongly generates $W$ if 
$$
W=span(:T^{j_1}(w_{1})\cdots T^{j_r}(w_{r}):\mid r\in\ZZ_+,\ w_i\in\mathfrak V).
$$ 
 
 Our first result is the following general criterion for the existence of a conformal embedding 
$$ \widetilde V(\mathfrak a) \hookrightarrow \overline W,$$  
where $\overline W=W/I$ is a nonzero quotient of $W$ and  $\widetilde V(\mathfrak a) =   V(\mathfrak a)  / (I \cap V(\mathfrak a) )$. It is clear that such an embedding exists if and only if $L-L^\aa$ generates a proper ideal of $W$.   
\begin{theorem}\label{Criterion}
Assume that $W$ is strongly generated by $\mathfrak a$ and by  $\{L-L^\aa\}\cup S$ with $S$ homogeneous with respect to the gradation \eqref{g2.1} and such that
 $(L-L^{\mathfrak a})(2)X=0$ for $X\in S\cap W(2)$.
Then $(L-L^{\mathfrak a})$ generates a  proper ideal $I$ in $W$ if and only if $c=c_\aa$.
\end{theorem}
\begin{proof}
As noted in the proof of Lemma \ref{ideala},
$
\langle\phi(a),\phi(b)\rangle=\overline{\langle a,b\rangle}$. In particular $\phi(\aa')=\aa'$, thus the ideal $I_{\aa'}$ is $\phi$--stable. Let $\pi_{\aa'}:W\to W/I_{\aa'}$ be the projection. It is clear that the ideal $I$ generated by $L-L^\aa$ is proper in $W$ if $\pi_{\aa'}(I)$ is proper in $W/I_{\aa'}$, we can therefore assume that $\aa'=\{0\}$, so that $k_i\ne 0$ for all $i$.
Moreover, if $\phi(\aa_i)=\aa_j$, then, since $k_i\in\R$, we have $k_i=k_j$ and for $a,b\in\aa_i$,  
\begin{equation}\label{<phi>}\langle \phi(a),\phi(b)\rangle_j=\overline{\langle a,b\rangle}_i.
\end{equation}

Let $\{x_i\}$ be a basis of $\aa$ and let $\{x^i\}$ its dual basis with respect to $\bigoplus\limits_{i}\langle\cdot,\cdot\rangle_i$. Then, by \eqref{<phi>}, 
$$\phi(\sum :x^ix_i:)=\sum :\phi(x^i)\phi(x_i):=\sum :x^ix_i:,
$$
so $\phi(L^{\aa})=L^{\aa}$.

Set $U=span(X\in S\mid X\in W(2))$. Observe that 
\begin{equation}\label{f1}W(2)=\C(L-L^{\aa})+U+V(\aa)\cap W(2).
\end{equation}
We have
\begin{equation}\label{f2}
(L-L^{\aa},U)=(\vac,(L^{\aa}-L^{\aa})(2)U)=0
\end{equation}
by our assumption that $
(L-L^{\aa})(2)U=0$.

If $c=c_{\mathfrak a}$ then
\begin{align}\notag
(L-L^{\aa},L-L^{\aa})&=
((L-L^{\aa})(2)(L-L^{\aa}),\vac)=\half \overline{(c-c_\aa)}=0.\label{f3}
\end{align}
Since $L-L^\aa\in Com(V(\aa),W)$,  we have
\begin{equation}\label{f4}
(L-L^\aa,V(\aa)\cap W(2))=0.
\end{equation}
From \eqref{f1}-\eqref{f4} it follows that $(L-L^\aa,W(2))=0$. Since $(W(i),W(j))=0$ if $i\ne j$ we see that $L-L^\aa$ is in the kernel of the form $(\cdot,\cdot)$, which is a proper ideal.

Converesely, if $L-L^{\aa}$ generates a proper ideal $I$ then $L-L^{\aa}=0$ in $W/I$, hence $(L-L^{\aa})(2)(L-L^{\aa})=\half(c-c_\aa)\vac=0$. This completes the proof.
\end{proof}

In the setting of Theorem \ref{Criterion}, we can clearly assume that  $S$ is a space  stable under the action of $\aa$. We make the further assumption that the action of $\aa$ on $W$ is completely reducible, thus we can choose $S$ so that $(\C L+\aa)\oplus S$ is a set of strong generators for $W$. Decomposing $S$ as $\aa$-module, we can write $S=\oplus_{i\in \mathcal J}S_i$ with $S_i$ irreducible and $\mathcal J$ some index set. Write 
$$
S_i=\bigotimes_j S_i^j
$$
with $S_i^j$ irreducible $\aa_j$--modules. Set
\begin{equation}\label{Cj}
C_i=\sum_{j\ge1}\frac{c^{(i)}_j}{2(k_j+h^\vee_j)}+(1-\d_{k_0,0})\frac{c^{(i)}_0}{2k_0},
\end{equation}
where $c^{(i)}_j$ is  the eigenvalue of $C_{\aa_j}$ on $S_i^j$. 
Let $M$ be a  weight $V(\aa)$--module such that $S_i\subset M^{top}$. Then 
\begin{equation}\label{Sisingular}
L^\aa(0)v=C_iv,\, v\in S_i\subset M^{top}.
\end{equation}
Recall that we are assuming that $S$ is homogeneous under the action of $L(0)$; since the actions of 
$\aa$ and $L(0)$ commute, we can clearly assume that also $S_i, i\in\mathcal J$ are homogeneous: we set $\D_i$ to  be the eigenvalue of $L(0)$ on $S_i$.
Note that we are not assuming that the action of $L^\aa_0$ on $W$ is semisimple.

In the setting  of Theorem \ref{Criterion}, assume that $c=c_\aa$ and let $I$ be the proper ideal generated by $L-L^\aa$. Let $\pi_I: W\to W/I=\overline W$ be the quotient map.
Since $L=L^\aa$ on $\overline W$, the action of $L^\aa$ is semisimple on $\overline W$. 
Choose as above a set of strong generators $(\C(L-L^\aa)+\aa)\oplus \sum_{i\in\mathcal J} S_i$.
Then $\overline W$ is strongly generated by $$\aa \oplus \sum_{i\in\mathcal K} \pi_I(S_i)$$
and we can assume $\mathcal K\subset \mathcal J$ minimal.
%to be minimal in the sense that, if $\mathcal K$ is  a proper subset of $\mathcal J$, then $\aa\oplus (\sum_{i\in \mathcal K}\pi_I(S_i))$ does not strongly generate $\overline W$. 
 
\begin{theorem}\label{Criterion2} With the above assumptions, $C_j= \D_j$ for all $j\in\mathcal K$. In particular, if $C_i\ne \D_i$ for all $i\in \mathcal J$, then 
$\overline W\cong\widetilde V(\aa)$.
\end{theorem}

\begin{proof}
%Check details
Fix $\D_i, i\in\mathcal K$. Let $\overline W'$ be the vertex subalgebra generated by 
$$
\overline W(\D_i)^-=\sum_{n<\D_i}\overline W(n).
$$
%{\color{red}  Is  $\overline W'$ a $V(\aa)$--submodule or maybe it  should be a  vertex subalgebra generated by $\overline W(\D_i)^-$?}
Note that, if $a\in \aa$ and $n\ge0$, then $a_{n}\overline W(\D_i)^-\subset  \overline W(\D_i)^-$.
It follows that $\overline W'$ is strongly generated by $\overline W(\D_i)^-$.
By  the minimality of $\mathcal K$, this observation implies that 
  $\pi_I(S_i)\cap \overline W'=\{0\}$. By construction, $a_{n}(\pi_I(S_i)+\overline W')=\overline W'$ %{\color{red} Should be $0$?} 
  in $\overline W/\overline W'$ if $n>0$ and $a\in\aa$, hence,  by \eqref{Sisingular},
 $$
 L^\aa(0)(\pi_I(S_i)+\overline W')=C_i(\pi_I(S_i)+\overline W').
 $$
 On the other hand, since $L=L^\aa$ in $\overline W$, we find
 $$
 L^\aa(0)(\pi_I(S_i)+\overline W')=L(0)(\pi_I(S_i))+\overline W'=\D_i(\pi_I(S_i)+\overline W').
 $$
 Since $\pi_I(S_i)+\overline W'\ne0$ the claim  follows.

\end{proof}

%\begin{remark}
 
% (1).  Observe that $\C(L-L^{\aa})+V(\aa)\cap W(2)$ is a $\aa-stable$ subspace of $W(2)$. If the trivial representation of $\aa$ does not occur in $W(2)/(\C(L-L^{\aa})+V(\aa)\cap W(2))$ then one can choose $S_0\cup S_1$ so that condition (b) of Theorem \ref{Criterion} is fulfilled.

%(2).  Condition (b) of Theorem \ref{Criterion} is obviously fulfilled if the generators in $(S_0\cup S_1)\cap W(2)$ are primary.

%\end{remark}
%\section{ A criterion for the semi-simplicity of decompositions}
% Should be written
\section{Structure}
We adopt the setting and notation of Section\ 1 of \cite{KW}. We let  $W^k(\g,x,f)$ be the universal $W$--algebra of level $k\in \R$ associated to the datum  $(\fg ,x,f)$, where  $\fg$ is  a basic Lie superalgebra, $x$ is an
$\ad$--diagonalizable element of $\fg$ with eigenvalues in
$\tfrac{1}{2}\ZZ$, $f$ is an even element of $\fg$ such that
$[x,f]=-f$ and the eigenvalues of $\ad x$ on the centralizer
$\fg^f$ of $f$ in $\fg$ are non-positive. We also assume that the datum $(\g,x,f)$ is {\it Dynkin}, i.e. there is a $sl(2)$-triple $\{e,h,f\}$ and $x=\half h$.
Let 
\begin{equation}\label{gg}\g=\bigoplus\limits_{j\in \frac{1}{2}\Z}\g_j
\end{equation}
 be the grading of $\g$ by $\ad(x)$--eigenspaces.
%This is the case, for example, of the minimal $W$--algebras $W^k(\g,\theta/2)$ coming from Table 2 in Section 5 of \cite{KW1}. 
We 
%also 
assume that $k\ne -h^\vee$ so that $W^k(\g,x,f)$ has a Virasoro vector.
Then
%are necessary and sufficient assumptions for 
$W^k(\g,x,f)$ is 
%to be 
a conformal vertex algebra in the sense of Definition \ref{svoa}. It is known that the vertex algebra structure of $W^k(\g,x,f)$ depends only on the $G_{\bar 0}$-orbit of $f$ (where $G_{\bar 0}$ is the adjoint group corresponding to $\g_{\bar 0}$), but the conformal structure does depend also on the grading \eqref{gg}.

We let:
\begin{equation}
  \label{eq:1.2}
  \fg_+ =\bigoplus_{j>0} \fg_j \, , \quad \fg_- = \bigoplus_{j<0} \fg_j
  \, , \quad \fg_{\leq} = \fg_0 \oplus \fg_- \, .
\end{equation}
The element $f$ defines a skew-supersymmetric even bilinear form
$\langle\, . \, , \, . \rangle_{ne}$ on $\fg_{1/2}$ by the
formula:
\begin{equation}
  \label{eq:1.3}
  \langle a , b \rangle_{ne} = (f | [a,b]) \, .
\end{equation}
Denote by $\fg^{\natural}$ the centralizer
of $f$ in $\fg_0$. 
%The bilinear form $\langle . \, , \,
%. \rangle_{ne}$ is invariant with respect to the representation of
%$\fg^{\natural}$ on
%%$\fg_{1/2}$:
%
%\begin{equation}
%  \label{eq:1.5}
%  \langle [v,a],b \rangle_{\ne} + (-1)^{p(v)p(a)}
%  \langle a,[v,b] \rangle_{\ne} =0 \quad
 % (v \in \fg^{\natural}, a,b \in \fg_{1/2}) \, .
%\end{equation}

Denote by $A_{ne}$ the vector superspace $\fg_{1/2}$ endowed
with the bilinear form (\ref{eq:1.3}).  Denote by $A$
(resp. $A^*$) the vector superspace $\fg_+$ (resp. $\fg^*_+$)
with the reversed parity, let $A_{{ch}} =A \oplus A^*$ and define
an even skew-supersymmetric non-degenerate bilinear form $\langle . \,
, \, . \rangle_{{ch}}$ on $A_{{ch}}$ by
\begin{eqnarray}
  \label{eq:1.6}
  \langle A,A \rangle_{{ch}} &=& 0 = \langle A^* ,A^* \rangle_{{ch}}
   \, , \, \\
\nonumber
\langle a,b^* \rangle_{{ch}} &=& -(-1)^{p(a)p(b^*)}
   \langle b^* ,a \rangle_{{ch}} =b^* (a) \hbox{ for }
   a \in A , b^* \in A^* \, .
\end{eqnarray}
%
%Here and further, $p(a)$ stands for the parity of an
%(homogeneous) element of a vector superspace.

%Following \cite{KRW}, introduce the differential complex $(\C (\fg
%,x,f,k),d_0)$, where $\C (\fg ,x,f,k)$  is a vertex algebra depending on a
%complex parameter $k$ and $d_0$ is an odd derivation of all
%products of this vertex algebra, such that $d^2_0 =0$.  We have:
%
%\begin{equation}
%  \label{eq:1.7}
%  \C (\fg , x,f ,k)=V_k (\fg) \otimes F (\fg , x,f) \, ,
%\end{equation}
%
%&where $V_k (\fg)$ is the \emph{universal affine vertex algebra of level}
%$k$ associated to $\fg$ and

%\begin{equation}
%\  \label{eq:1.8}
 %\ F(\fg , x,f) =F (A_{{ch}}) \otimes F (A_{\ne}) \, ,
%\\end{equation}

Let  $F(A_{{ch}})$ and $F(A_{ne})$ be  the fermionic vertex algebras based on $A_{{ch}}, A_{ne}$ respectively. Recall from \cite{KW} that $W^k(\g,x,f)$  is a vertex subalgebra of
$V^k(\g)\otimes F(A_{{ch}})\otimes F(A_{ne})$.

Choose a basis $\{
u_{\alpha} \}_{\alpha \in \Sigma_j}$ of each $\fg_j$ in
\eqref{gg}, and let $\Sigma=\coprod_{j \in \tfrac{1}{2} \ZZ} \Sigma_j$,
$\Sigma_+ = \coprod_{j>0} \Sigma_j$.  Let $m_{\alpha} =j$
if $\alpha \in \Sigma_j$.  
%Define the structure constants
%$c^{\gamma}_{\alpha \beta}$ by $[u_{\alpha}, u_{\beta}] =
%\sum_{\gamma} c^{\gamma}_{\alpha \beta} u_{\gamma}$ $(\alpha,
%\beta, \gamma \in S)$.  
Denote by $\{ \varphi_{\alpha}\}_{\alpha
  \in \Sigma_+}$ the corresponding basis of $A$ and by $\{
\varphi^{\alpha}\}_{\alpha \in \Sigma_+}$ the basis of $A^*$ such that
$\langle \varphi_{\alpha}, \varphi^{\beta} \rangle_{ch}
=\delta_{\alpha \beta}$.  Denote by $\{ \Phi_{\alpha}\}_{\alpha
  \in \Sigma_+}$ the corresponding basis of $A_{ne}$, and by $\{
\Phi^{\alpha} \}_{\alpha \in \Sigma_{1/2}}$ the dual basis
with respect to $\langle . \, , \, . \rangle_{ne}$, i.e.,~$\langle
\Phi_{\alpha} , \Phi^{\beta} \rangle_{ne} = \delta_{\alpha \beta}$.
  It will also be convenient to define $\Phi_u$ for any $u \in
  \sum_{\alpha \in \Sigma} c_{\alpha}u_{\alpha} \in \fg$ by letting
  $\Phi_u =\sum_{\alpha \in \Sigma_{1/2}} c_{\alpha}\Phi_{\alpha}$; similarly, for $u\in \g_+$, we will use the notation $\varphi_u$.

\begin{rem}
Since we have chosen the grading to be Dynkin, we can apply \cite[Lemma 7.3]{KMP}. Thus a conjugate linear involution $\phi$ of $\g$ fixing $x$ and $f$ and satisfying $(\phi(a)|\phi(b))=\overline{(a|b)}$ descends to define a conjugate linear involution of $W^k(\g,x,f)$ that we still denote by $\phi$.
\end{rem}
Set
$$
\widehat{\g^f}=\bigoplus_{j\le0}(\g^f_j\otimes \C[t^{-1}]t^{-1+j}).
$$
The space $\widehat{\g^f}$ admits a grading by the action of $D=-t\frac{d}{dt}$. Moreover there is a natural action of $\g^\natural$ on $\widehat{\g^f}$. Let $S(\widehat{\g^f})$ be the symmetric algebra of $\widehat{\g^f}$, equipped with the natural action of $\g^\natural$ and the action of $D$ extended by derivations. We let $D$ act on $W^k(\g,x,f)$ by the action of $L(0)$. Recall that $W^k(\g,x,f)(1)$ is isomorphic to $\g^\natural$ as a Lie superalgebra, so we can define an action of $\g^\natural$ on $W^k(\g,x,f)$ by letting $a\in \g^\natural \cong W^k(\g,x,f)(1)$ act by $a_0$.

We say that a finite--dimensional sub--superspace $\mathfrak V$ of $W$ is free in $W$ if there is a basis $\mathcal{B}=\{w_1,\ldots, w_s\}$ of $\mathfrak V$ such that the set
$$
\mathcal M(\mathcal B)=\left\{:T^{j_1}(w_{i_1})\cdots T^{j_r}(w_{i_r}):\mid r\in\ZZ_+,\ i_1\le i_2\cdots\le i_r, i_j<i_{j+1}\text{ if $p(w_{i_j})=1$}\right\} 
$$
is linearly independent. Obviously, if $\mathcal M(\mathcal B)$ is linearly independent for a basis $\mathcal B$ of $\mathfrak V$, then $\mathcal M(\mathcal B')$ is linearly independent for any basis $\mathcal B'$ of $\mathfrak V$.
 We say that $\mathfrak V$ strongly and freely generates $W$ if it is free and strongly generates $W$.

\begin{theorem} \label{structure} Assume that $W^k(\g,x,f)$ is completely reducible as a $\g^\natural$--module. Then there is a $\C D\oplus\g^\natural$--module isomorphism $\Psi:S(\widehat{\g^f})\to W^k(\g,x,f)$ such that  $\Psi(\g^f)$ strongly and freely  generates $W^k(\g,x,f)$.
Moreover one can choose $\Psi$ so that $\Psi(f)=L$.
\end{theorem}
\begin{proof}
The existence of a  $\g^\natural$--module isomorphism $\Psi:S(\widehat{\g^f})\to W^k(\g,x,f)$ is highlighted in \cite[Proposition 3.1]{CL} as a corollary of the proof of Theorem 4.1 in \cite{KW}.  More precisely, Theorem 4.1 in \cite{KW} provides a vector space isomorphism between $S(\widehat{\g^f})$ and  $W^k(\g,x,f)$, mapping a polynomial over $\widehat{\g^f}$ into the normally ordered product of the corresponding generators of $W^k(\g,x,f)$. 
In \cite[Proposition 3.1]{CL} it is observed that it is possible to modify $\Psi$ into a $\g^\natural$-module map. 
We observe  that the fact that one can choose $\Psi$ so that $\Psi(\g^f)$ strongly and freely  generates $W^k(\g,x,f)$ is also implicit in the construction given in Theorem 4.1 in \cite{KW}. Here we give a proof of this latter statement.

Recall from \cite{KW} that there is a $\C D$--module  monomorphism $\Psi':\g^f\to W^k(\g,x,f)$ such that $W^k(\g,x,f)$ is strongly and freely generated by $\Psi'(\g^f)$. We prove by induction on $j$ that there is a $\C D$--module monomorphism  $\Psi_j: \g^f\to W^k(\g,x,f)$ such that $\Psi_j$ restricted to $\bigoplus\limits_{i=0}^{j-1}\g^f_{-i}\otimes t^{-i-1}$ is a $\C D\oplus\g^\natural$--module homomorphism and such that  $W^k(\g,x,f)$ is strongly and freely generated by $\Psi_j(\g^f)$.

For $j=1$ it is enough to choose $\Psi_1=\Psi'$. We now assume by induction the existence of $\Psi_j$ and construct $\Psi_{j+\frac{1}{2}}$.

For each $r$ choose a basis $\{a^r_i\}$ of $\g^f_r$ and choose an order for
$$
\widehat{\mathcal B}=\bigcup_{s,i,r}\{a^r_i\otimes t^{-s+r-1}\}.
$$
so that we can write  $\widehat{\mathcal B}=\{b_1,b_2,\ldots,b_i,\ldots\}$. Clearly $\widehat{\mathcal B}$ is an ordered basis for $\widehat{\g^f}$.
Extend $\Psi_j$ to $\widehat{\g^f}$ by setting
$\Psi_j(a^r_i\otimes t^{-s+r-1})=T^s(\Psi_j(a^r_i))$ and also to $S(\widehat{\g^f})$ by mapping the ordered monomials $b_{i_1}\cdots b_{i_r}$ to
$
:\Psi_j(b_{i_1})\cdots \Psi_j(b_{i_r}):$.
 By the induction hypothesis this extension of $\Psi_j$ is an isomorphism of $\C D$--modules.
Write 
$$
S(\widehat{\g^f})=\bigoplus_{n\in \mathbb Z_+} S(\widehat{\g^f})(n)
$$
for the grading defined by the action of $D$. Set 
$$
U=\Psi_j\left(S\left(\bigoplus_{i=0}^{j-1}(\g^f_{-i}\otimes \C[t^{-1}]t^{-1-i})\right)(j+\tfrac{1}{2})\right).
$$
Since $\Psi_j$ is  $\g^\natural$--equivariant when restricted to $S\left(\bigoplus_{i=0}^{j-1}(\g^f_{-i}\otimes \C[t^{-1}]t^{-1-i})\right)$, it follows that $U$ is $\g^\natural$--stable. Let $V$ be a $\g^\natural$--stable complement of $U$ in $W^k(\g,x,f)(j+\tfrac{1}{2})$.

Since, by \cite[Proposition 3.1]{CL},  $S(\widehat{\g^f})(j+\tfrac{1}{2})$ and $W^k(\g,x,f)(j+\tfrac{1}{2})$ are isomorphic as $\g^\natural$--modules, we have
$$
V\cong W^k(\g,x,f)(j+\tfrac{1}{2})/U \cong S(\widehat{\g^f})(j+\tfrac{1}{2})\bigg/S\left(\bigoplus_{i=0}^{j-1}(\g^f_{-i}\otimes \C[t^{-1}]t^{-1-i})\right)(j+\tfrac{1}{2})\cong \g^f_{-j+\frac{1}{2}}.
$$
Let $\phi:\g^f_{-j+\frac{1}{2}}\to V$ be a $\g^\natural$--module isomorphism. 
We define $(\Psi_{j+\frac{1}{2}})_{|\g^f_{-j+\frac{1}{2}}}=\phi$ and 
$(\Psi_{j+\frac{1}{2}})_{|\g^f_{i}}=\Psi_j$ for $i\ne -j+\frac{1}{2}$.

It remains to check that $\Psi_{j+\frac{1}{2}}(\g^f)$ strongly and freely generates $W^k(\g,x,f)$.

Since $\Psi_j(\g^f)$ is free, the projection $p_V:\Psi_j(\g^f_{-j+\frac{1}{2}})\to V$ with respect to the decomposition $W^k(\g,x,f)(j+\frac{1}{2})=V\oplus U$ is injective, hence, since $\Psi_j(\g^f_{-j+\frac{1}{2}})$ and $V$ have the same dimension, also bijective. 
In particular the set $\{\widetilde a^{-j+\frac{1}{2}}_i\}$ such that
$$
\{\widetilde a^{-j+\frac{1}{2}}_i\otimes t^{-j-\frac{1}{2}}\}=\phi^{-1}(\{p_V(\Psi_j(a^{j-\frac{1}{2}}_i\otimes t^{-j-\frac{1}{2}}))\})
$$
 is a basis  of $\g^f_{-j+\frac{1}{2}}$. Let $\widetilde{\mathcal B}=\{\widetilde b_1,\widetilde b_2,\ldots, \}$ be the basis of $S(\widehat{\g^f})$ constructed as $\widehat{\mathcal B}$ using the basis $\{\widetilde a^{-j+\frac{1}{2}}_i\}\cup (\bigcup_{r\ne -j+\frac{1}{2}}\{a^r_i\})$ of $\g^f$.
 
 We need to show that the monomials 
 $$:\Psi_{j+\frac{1}{2}}(\widetilde b_{i_1})\cdots \Psi_{j+\frac{1}{2}}(\widetilde b_{i_r}):$$
 form a basis of $W^k(\g,x,f)$. To this end, following \cite{DK}, we grade $\widehat{\g^f}$ by setting $\deg(a^{r}_i\otimes t^s)=r+1$ and extend the grading to $S(\widehat{\g^f})$. Set  
 $$
 S(\widehat{\g^f})_p=\{a\in  S(\widehat{\g^f})\mid \deg(a)\le p\}.
 $$
Set also  $W^k(\g,x,f)_p=\Psi_j( S(\widehat{\g^f})_p)$, thus obtaining  a filtration of $W^k(\g,x,f)$, which, according to \cite{DK} and Example 3.12 of \cite{DKF}, has the property that
$$
:W^k(\g,x,f)_pW^k(\g,x,f)_q:\subset W^k(\g,x,f)_{p+q},
$$
and
$$[(W^k(\g,x,f)_p)_\lambda (W^k(\g,x,f)_q)]\subset W^k(\g,x,f)_{p+q-\frac{1}{2}}.
$$ 
It follows from (1.39) and (1.40) of \cite{DKF}, that the normal order $:\cdot :$ induces a (super)--commutative and (super)--associative product on $\mbox{gr}(W^k(\g,x,f))$.
We define a grading of the resulting algebra $\mbox{gr}(W^k(\g,x,f))$ by giving degree $1$ to $\Psi_j(a^{-j+\frac{1}{2}}_i\otimes t^s)$ and degree $0$ to $\Psi_j(\oplus_{i\ne- j+\frac{1}{2}}(\g^f_i\otimes \C[t^{-1}]t^{-1+i}))$.
Write $\mbox{gr}(W^k(\g,x,f))^n$ for the space of degree $n$ in
$\mbox{gr}(W^k(\g,x,f))$.
By construction we have that 
$$
\Psi_{j+\frac{1}{2}} (\widetilde b_{i})=\Psi_j(b_{i})+\sum c^i_{i_1,\ldots, i_s }:\Psi_j(b_{i_1})\cdots \Psi_j(b_{i_s}):
$$
with $\deg(b_{i_t})<j+\tfrac{1}{2}$, so, if
$\deg(:\Psi_{j+\frac{1}{2}}(\widetilde b_{i_1})\cdots \Psi_{j+\frac{1}{2}}(\widetilde b_{i_r}):)=m$ and
$:\Psi_{j+\frac{1}{2}}(\widetilde b_{i_1})\cdots \Psi_{j+\frac{1}{2}}(\widetilde b_{i_r}):+(W^k(\g,x,f))_{m-\frac{1}{2}}\in \mbox{gr}(W^k(\g,x,f))^n
$
then
\begin{align*}
:\Psi_{j+\frac{1}{2}}&(\widetilde b_{i_1})\cdots \Psi_{j+\frac{1}{2}}(\widetilde b_{i_r}):+(W^k(\g,x,f))_{m-\frac{1}{2}}\\
&=:\Psi_{j}( b_{i_1})\cdots \Psi_{j}( b_{i_r}):+(W^k(\g,x,f))_{m-\frac{1}{2}}\mod \mbox{gr}(W^k(\g,x,f))^{n-1}.
\end{align*}
This concludes the induction step.

It remains only to check that we can set up $\Psi$ so that $\Psi(f)=L$. This is achieved by choosing in the construction of $\Psi_2$ a $\g^\natural$--stable complement $V'$ of $\C L\oplus \Psi_{\frac{3}{2}}(S(\g^f_0\otimes \C[t^{-1}]t^{-1})(2))$, a $\g^\natural$--stable complement $V''$ of $\C f$ in $\g^f_1$, and defining $\phi$ by setting $\phi(f)=L$ and $\phi_{|V''}=\phi'$ with $\phi'$ any $\g^\natural$--isomorphism between $V''$ and $V'$.
\end{proof}
%{\color{blue}Up to minor details Theorem 3.1 should be complete}

\section{Hook type $W$-algebras}

We shall consider the special case $\g ={sl}(m+n)$, and the nilpotent element $f= f_{m,n}$ determined by the partition $(m,1^n)$.
In this case $\g^{\natural} \cong{gl}(n)$.  The  labels of the weighted Dynkin diagram corresponding to a dominant semisimple element in the orbit of $(x,f)$ are 
\begin{align}
&\text{$m$ even:}\quad  (\underbrace{1,\ldots,1}_{\tfrac{m}{2}-1},\tfrac{1}{2},\underbrace{0,\ldots,0}_{n-1},\tfrac{1}{2},\underbrace{1,\ldots,1}_{\tfrac{m}{2}-1}),\label{Dyneven}\\
&\text{$m$ odd:}\quad  (\underbrace{1,\ldots,1}_{\tfrac{m-1}{2}},\underbrace{0,\ldots,0}_{n},\underbrace{1,\ldots,1}_{\tfrac{m-1}{2}}).\label{Dynodd}
\end{align}
\begin{rem}\label{good} Note that when $m$ is even the grading is not even; on the other hand,  using Kac-Elashivili theory of good gradings \cite{EK}, one easily checks that the semisimple element
$$x_{good}=diag\left(\tfrac{(m+2n)(m-1)}{2(m+n)},\tfrac{(m+2n)(m-1)}{2(m+n)}-1,\ldots,-\tfrac{m(m-1)}{2(m+n)},\underbrace{-\tfrac{m(m-1)}{2(m+n)},\ldots,-\tfrac{m(m-1)}{2(m+n)}}_{n}\right)$$
gives rise to an even good grading for $f_{m,n}$.
\end{rem}
From \eqref{Dyneven}, \eqref{Dynodd}  it is easy to compute the dimension of $\g_j$, hence also those of $\g_j^f$, since 
\begin{equation}\label{calcolodim}\dim \g_j^f=\dim \g_j-\dim\g_{j-1}\text{ for $j\le 0$}.\end{equation} These data are summed up in  Tables 1,2.

\begin{table} 
\caption{$m$ odd }
\vskip 5pt
\begin{tabular}{c|c | c}
 & $\dim\mathfrak g_{-j}$ &$\dim\mathfrak g_{-j}^f$\\
\hline
$\frac{m+1}{2}\leq j\leq m-1$&$m-j$ &$1$\\\hline
$j=\frac{m-1}{2}$&$2n+\frac{m+1}{2}$&$2n+1$\\\hline
$1\leq j\leq\frac{m-3}{2}$&$2n+m-j$&$1$\\\hline
$j=0$&$n^2+2n+m-1$&$n^2$\\
\end{tabular}\vskip 5pt\vskip 5pt\vskip 5pt
\caption{$m$ even }
\vskip 5pt
\begin{tabular}{c|c | c}
 & $\dim\mathfrak g_{-j}$ &$\dim\mathfrak g_{-j}^f$\\
\hline
$1\leq j\leq m-1$&$m-j$ &$1$\\\hline
$j=\tfrac{m-1}{2}$&$2n$&$2n$\\\hline
$j=i+\frac{1}{2},\, 0\leq i <\frac{m}{2}-1$&$2n$&$0$\\\hline
$0$& $n^2+m-1$&$n^2$\\
\end{tabular}
\end{table}

\begin{theorem}\label{41} Assume that $k$ is not critical, i.e. $k \ne -n-m $. 
One can choose strong generators for the vertex algebra  $W^k(\g,x, f_{m,n})$ as follows: 
\begin{itemize}
\item[(1)]  $J^{\{a\}},\, a\in  \g^{\natural}\cong gl(n)$; these generators are primary  for $L$ of conformal weight $1$;
\item[(2)]the Virasoro field $L$;
\item[(3)]   fields $W_i$, $i= 3,  \dots, m$, of conformal weight $i$;%which are quasi-primary.
\item[(4)]fields $G^\pm _i, i=1, \dots, n $ of conformal weight $\tfrac{m+1}{2}$.
\end{itemize}
The fields $G^\pm _i, i=1, \dots, n $  are primary for both $L$ and $V (\g^{\natural})$.  As $sl(n)$--modules,
$$\mbox{span}_{\mathbb C} \{G^+ _i, i=1, \dots, n\}\cong \C^n, \quad
 \mbox{span}_{\mathbb C} \{G^- _i, i=1, \dots, n\}\cong (\C^n)^*.$$
 Finally, $gl(n)$ acts trivially on $W_i$.
\end{theorem}
%{\color{red} Can we assume that $W_i$ are $M(1)$--primary, where $V(\g^{\natural})  = V(sl(n)) \otimes M(1)$?}
%{\color{blue} Yes, provided $k\ne-(m+n)(1-\tfrac{1}{m})$}
\begin{proof}We prove statements (1)--(4) by applying  Theorem \ref{structure}, hence we need to describe explicitly the structure of $\g^f$ as a $\g^\natural$--module. We choose
\begin{equation}\label{xf}
f=\left(\begin{array}{c|c}
  J_{m} & 0 \\
\hline
 0 & 0 
\end{array}\right),\quad x=\left(\begin{array}{cccc|c}
  \tfrac{m-1}{2}&0&\cdots&0 & \\
0&\tfrac{m-3}{2}&\cdots&0 &  \\
\vdots&\vdots&\ddots&\vdots&0\\
 0&0&\cdots &-\tfrac{m-1}{2} &  \\
 \hline
 &&0&&0
\end{array}\right).
\end{equation}
Here  $J_m$ the $m\times m$ is the Jordan block
$$
J_m=\left(\begin{array}{ccccc}
 0&0&\cdots&0&0 \\
1&0&\cdots&0&0  \\
0&1&\cdots&0&0  \\
\vdots&\vdots&\ddots&\vdots&\vdots\\
 0&0&\cdots &1 &0
 \end{array}\right).
$$
In our case 
\begin{equation}\g^\natural =\left\{\left(
 \begin{array}{c|c}
-\tfrac{ tr(A)}{m}I_m& 0 \\
\hline
 0 & A 
 \end{array}\right)
\mid A\in gl(n)\right\}\cong gl(n),
\end{equation}
and the  1-dimensional center is spanned by \begin{equation}\label{varpi}\varpi=\left(
 \begin{array}{c|c}
-\tfrac{n}{m}I_m& 0 \\
\hline
 0 & I_n
 \end{array}\right).\end{equation}

If $m$ is odd
$$
\g^f_j=\C \left(\begin{array}{c|c}
J^j_m&0\\
\hline
0&0
 \end{array}\right),\ 1\le j\le m-1, j\ne \frac{m-1}{2},
 $$
 and
 $$
\g^f_{(m-1)/2}=\C \left(\begin{array}{c|c}
J^{(m-1)/2}_m&0\\
\hline
0&0
 \end{array}\right)\oplus \left\{\left(\begin{array}{c|c|c}
0&0&0\\
\hline
0&0&v\\
\hline
{}^t w&0&0
 \end{array}\right)\mid v,w\in\C^n\right\}.
 $$
 
 If $m$ is even
 $$
\g^f_j=\C \left(\begin{array}{c|c}
J^j_m&0\\
\hline
0&0
 \end{array}\right),\ 1\le j\le m-1,
 $$
 and
 $$
\g^f_{(m-1)/2}= \left\{\left(\begin{array}{c|c|c}
0&0&0\\
\hline
0&0&v\\
\hline
{}^t w&0&0
 \end{array}\right)\mid v,w\in\C^n\right\}.
 $$
 Note that 
 \begin{equation}\label{actionvarpi}
\left[\varpi,\left(\begin{array}{c|c|c}
0&0&0\\
\hline
0&0&v\\
\hline
{}^t w&0&0
 \end{array}\right)\right]=\left(\begin{array}{c|c|c}
0&0&0\\
\hline
0&0&-\frac{n+m}{m}v\\
\hline
\frac{n+m}{m}\,{}^t w&0&0
 \end{array}\right).
 \end{equation}
 Now we check that $\g^\natural$ acts semisimply on  $ W^k(\g,x,f_{m,n})$. Since $W^k(\g,x,f_{m,n})(s)$ is  finite dimensional for every $s$, $[\g^\natural, \g^\natural]$ acts semisimply on $W^k(\g,x,f)$. To analyze the action of  the center $\C\varpi$ on $V^k(\g)\otimes F(A_{ch})\otimes F(A_{ne})$ we use the explicit formula given in 
 \cite[(2.4), (2.7), Theorem 2.1 (a)]{KW}
 $$
 J^{\{\varpi\}} = \varpi + \sum_{\beta \in \S_+}: \varphi_{[\varpi,u_\beta]}
  \varphi^{\beta}  : + \tfrac{1}{2}
        \sum_{\alpha \in \Sigma_{1/2}}
        : \Phi^{\alpha}\Phi_{[u_{\alpha},\varpi]}:.
$$
If $v\in \g_{1/2}$, then  
\begin{align*}
\tfrac{1}{2}
        \sum_{\alpha \in \Sigma_{1/2}}&
        [: \Phi^{\alpha}\Phi_{[u_{\alpha},\varpi]}:_\l v]=-\tfrac{1}{2}
        \sum_{\alpha \in \Sigma_{1/2}}
        [v_{-\l-T}: \Phi^{\alpha}\Phi_{[u_{\alpha},\varpi]}:]\\
&=-\tfrac{1}{2}
        \sum_{\alpha \in \Sigma_{1/2}}
        \langle v, u^{\alpha}\rangle_{ne}\Phi_{[u_{\alpha},\varpi]}-\tfrac{1}{2}
        \sum_{\alpha \in \Sigma_{1/2}}
        \Phi^{\alpha} \langle v,[u_{\alpha},\varpi]\rangle_{ne}=\Phi_{[\varpi,v]},
\end{align*}
hence the monomials $:T^{j_1}\Phi_{\a_1}\cdots T^{j_s}\Phi_{\a_s}:$ generate $F(A_{ne})$ and are eigenvectors for 
$$\tfrac{1}{2}
        \sum_{\alpha \in \Sigma_{1/2}}
        : \Phi^{\alpha}\Phi_{[u_{\alpha},\varpi]}:_0.
        $$
          Similarly, if $v\in\g_+$,
$$
        \sum_{\beta \in \S_{+}}
        [: \varphi_{[\varpi,u_{\be}]}\varphi^{\be}:_\l v]=\varphi_{[\varpi,v]},
$$
hence   $ \sum_{\beta \in \S_{+}}
        : \varphi_{[\varpi,u_{\be}]}\varphi^{\be}: _0$ acts semisimply on $F(A_{ch})$. Since $\varpi_0$ acts semisimply on $V^k(\g)$, we conclude that    $ J^{\{\varpi\}}_0$ acts semisimply on $V^k(\g)\otimes F(A_{ch})\otimes F(A_{ne})$, hence also on $W^k(\g,x,f_{m,n})$.
        We can therefore  apply Theorem \ref{structure}: we set
$$
W_i=\Psi(\left(\begin{array}{c|c}
J^{i-1}_m&0\\
\hline
0&0
 \end{array}\right)),\ 3\le i\le m,
 $$
 and
 $$
 G^{+}_i=\Psi(E_{m+i,1}),\ 1\le i\le n,\ G^{-}_i=\Psi(E_{m,m+i}),\ 1\le i\le n.
 $$
 It remains to check that $G^\pm_i$ are primary: by the first part of the proof, $ J^{\{\varpi\}}_0$ acts trivially on $W^k(\g,x,f)(n)$ for $n<\frac{m+1}{2}$. Since $J^{\{\varpi\}}_0G^{\pm}_i=\pm G^{\pm}_i$ for all $i$ and $ [J^{\{\varpi\}}_0,L(n)]=0$, we conclude that, if $n>0$, 
 $$\pm L(n)G^\pm_i=J^{\{\varpi\}}_0L(n)G^\pm_i=0.
 $$
 The same argument shows that, if $a\in\g^\natural$, then 
 $a_{n}G^\pm_i=0$ if $n>0$.
\end{proof}
\begin{rem} In  \cite[Theorem 9.5]{CL} the authors  prove the  following stronger result:  hook type $W$-algebras are unique, under certain non-degeneracy conditions. 
\end{rem}
Let $(\cdot\,|\,\cdot)$ denote the trace form on $sl(n+m)$. The Killing form of $\g$ is $K_\g(\cdot\, , \, \cdot )=2h^\vee(\cdot\,|\,\cdot)$ with $h^\vee=n+m$.
Recall from  \cite{KW}  that for any $W$-algebra the central charge is given by the formula
\begin{equation}\label{cc}
 c(\g ,x,f,k) =\tfrac{k\dim \g}{k+h^\vee}
  -12 k (x|x)
- \sum_{\alpha \in S_+} 
   (12 m^2_{\alpha}-12m_{\alpha}+2)-\tfrac{1}{2}
   \dim \g_{1/2}.
\end{equation}
where $m_\a=j$ if $\a\in S_j$. 
In the hook case we write $c_{m,n}(k)$ for $c(\g ,x,f_{m,n},k)$. To calculate explicitly $c_{m,n}(k)$ note that we have: 
$$(x|x)=\frac{1}{2}\binom{m+1}{3}.$$
 Moreover, using the data in Table 1, one computes that the contribution of the last two terms in \eqref{cc} is 
{\small
$$\sum_{j=1}^{(m-3)/2}(12 j^2-12 j+2)(2n+m-j) +(3(m-1)^2-6(m-1)+2))(2n+\tfrac{m+1}{2})
+\sum_{j=(m+1)/2}^{m-1}(12 j^2-12j+2)(m-j)
$$}
if $m$ is odd  and 
{\small
$$\sum_{j=1}^{m-1}(12 j^2-12 j+2)(m-j) +2n(3(m-1)^2-6(m-1)+2)+2n\!\sum_{j=0}^{m/2-2}(3 (2j+1)^2-6(2j+1)+2)
+n
$$}
if $m$ is even. In any case, the following formula holds (see \cite[Appendix C]{XY}).

%For a derivation of the formula see \cite[Appendix C]{XY}.

%\begin{lemma}  The central charge of $W^k(\g, f)$ is given by the following formula:
{\small
\begin{equation}\label{cc1}c_{m,n}(k)=- \frac{   k + k (1 - m - n) (m + n) + (m + n)^2 }{k+m+n} +  m ( k - m - n - m^2 (1 + k + m + n) + m (1 + 3 (m + n))). \end{equation}}
%\end{lemma}

Let  $\widetilde{W}_k(\g, x,  f)$ be a quotient of $ W^k(\g, x,  f)$. We have the following homomorphisms:
\begin{equation}\label{immaff}V(\g^{\natural}) \rightarrow  W^k(\g, x, f) \rightarrow  \widetilde{W}_k(\g, x, f).\end{equation}  Let us denote the image of the resulting homomorphism by $\mathcal V(\g^{\natural})$.

%The next Lemma gives a necessary condition for an embedding to be conformal.     
\begin{theorem}\label{MT}
The embedding $\mathcal V(\g^{\natural}) \hookrightarrow W_k(\g, x,f_{m,n})$ is conformal if and only if $$k=k_{m,n} ^{(i)},\quad 1\le i\le 4,$$ where
\begin{itemize}
\item[(1)] $k_{m,n} ^{(1)}  = -n-m + \frac{n+m}{m+1}=-\frac{m}{m+1} h^{\vee}$ and $n >1$,
\item[(2)] $k_{m,n}^{(2)} =  -n -m + \frac{1 + m   + n }{m}=
-\frac{(m-1) h^{\vee} -1}{m}$ and $n \ge 1$,
\item[(3)] $ k_{m,n}^{(3)}  =\frac{-1 + 2 m - m^2 + 2 n - m n}{m-1} = - \frac{(m-2) h^{\vee}+1}{m-1} = -h^{\vee} + \frac{h^{\vee}-1}{m-1}  $  and $n \ge 1$, $m >1$,
\item[(4)] $k_{m,n}^{(4)} = -\frac{(m-1)h^{\vee}}{m} = -h^{\vee} + \frac{h^{\vee}}{m}$.
\end{itemize}
\end{theorem}
\begin{rem}
In the following cases we already proved that the above embeddings are conformal:
\begin{itemize}
\item   $m=1$, $k=k_{m,n} ^{(1)} = -\frac{n+1}{2}$,  since then $W_k(\g, x, f_{m,n}) = V_k(sl(n+1))$ (cf. \cite{AKMPP}).
\item $m=1$, $k=k_{m,n} ^{(2)} = 1$, since then  $W_k(\g, x, f_{m,n}) = V_1(sl(n+1))$.

\item   $m=2$,  $k =k_{m,n} ^{(1)} =-\tfrac{2h^{\vee}}{3}$ since  then $W_k(\g,f_{m,n}) = W_k(sl(n+2), f_{\theta})$ and $k$ is a conformal level (cf. \cite{AKMPP-JA},  \cite{AKMPP-JJM}).

\item   $m=2$,  $k =k_{m,n} ^{(2)} = -\frac{h^{\vee}-1}{2}$ since  then $W_k(\g,f_{m,n}) = W_k(sl(n+2), f_{\theta})$  and $k$ is conformal level (cf. \cite{AKMPP-JA},  \cite{AKMPP-JJM}).

\item   $m=2$,  $k =k_{m,n} ^{(3)} = -1$ since  then $W_k(\g,f_{m,n}) = W_k(sl(n+2), f_{\theta})$  and $k$ is a collapsing  level (cf. \cite{AKMPP-JA},  \cite{AKMPP-JJM}).
\item $m=2$, $k =k_{m,n} ^{(4)} = -\frac{h^{\vee}}{2}$  is collapsing level for $W_k(\g,f) = W_k(sl(n+2), f_{\theta})$.
\end{itemize}
\end{rem}
 
%\begin{rem}\label{Creu} T. Creutzig in \cite{TC} conjectured that 
%  the embedding 
%$\mathcal V (\g^{\natural}) \hookrightarrow W_k(\g, x, f_{m,n})$ is conformal when  $ k=k_{m,n} ^{(1)}$.
%\end{rem}

\begin{proof}[Proof of Theorem \ref{MT}] 
We  apply Theorem \ref{Criterion}.
We choose the bilinear form on $\g_0^\natural$ 
setting $\langle \varpi,\varpi\rangle_0= (\tfrac{n}{m})^2+1$.  We choose the conjugate linear involution $\phi$ to be the conjugation with respect to $sl(m+n,\R)$, so that 
$\phi(x)=x, \phi(L)=L$. 
 We choose 
 \begin{equation}\label{S} S= span\left(\{W_i\mid 3\le i\le m-1\}\cup\{G^\pm_i\mid 1\leq i\leq n\}\right).
\end{equation}
If $(m+1)/2\ne 2$, i.e. $m\ne 3$, then $S\cap W^k(\g,x,f_{m,n})(2)=\emptyset$, hence the hypotheses of Theorem \ref{Criterion} are vacuously verified. If $m=3$, we have to check that
$(L-L^\a)(2)G^\pm_i=0$: this readily follows from Theorem \ref{41}.\par
By Theorem \ref{Criterion},  it is enough  to equate the central charge \eqref{cc1} of $W_k(\g, x,f_{m,n})$ and the central charge  of $\mathcal V (\g^{\natural})$. By \cite[Theorem 2.1 (c)]{KW}, 
we have  that $\mathcal V(\g^{\natural})$ is a quotient of  $V^{k_0}(\C\varpi)\otimes V^{k_1}(sl(n))$, where 
$$k_0= k+ \frac{(m-1)(m+n)}{m},\quad  k _1 = k + m-1.$$
Moreover  $h_1^\vee=n, h_0^\vee=0$. We obtain the equations 
\begin{equation}\label{ecc} c_{m,n}(k)= \frac{k _1(n^2-1)}{k_1+n}+(1-\d_{{k_0},0}).\end{equation}
Solving for $k$, we get
$k= k_{m,n} ^{(i)}, 1\leq i\leq 3$  if $k_0\ne 0$. If instead $k_0=0$, then $k=k_{m,n} ^{(4)}$ and one readily verifies that this value of $k$ satisfies \eqref{ecc}.
\end{proof}
\section{Collapsing levels}
% {\color{blue}
  Assume that $k$ is a  conformal level. Then $ \bar L= L - L^{\g^{\natural}}$ belongs to the maximal ideal of   $W^k(\g, x,  f)$.  We define
  % $$\overline W_k(\g, x, f)  =  W^k(\g, x,  f)\cdot (L - L^{\g^{\natural}}). $$
   $$\overline W_k(\g, x, f)  = W^k(\g, x,  f)/ W^k(\g, x,  f)\cdot(L - L^{\g^{\natural}}). $$
Recall from the introduction that   a level $k$ is  called {\sl strongly collapsing} if $\mathcal V(\g^\natural)= \overline W_k(\g,x,f)$ (see \eqref{immaff}).

\begin{remark}  \label{ex-nstr}Clearly,  any strongly collapsing level is also a collapsing.  In \cite[Example 4.4]{AKMPP-JA}, we showed  that can exist collapsing levels which are not strongly collapsing. In particular, this holds for minimal affine $W$--algebras in the following cases:
\begin{itemize}
\item $\g = sl(3)$, $k=-1$; 
\item$\g=sp(4)$, $k=-2$.
\end{itemize}
\end{remark}

%}
%this definition has been introduced and analyzed in  \cite{AKMPP-JA} when $f$ is a minimal 
%nilpotent element and extended to the general case in \cite{AM}.
 \begin{theorem}\label{coll} For the hook $W$-algebra $W(sl(n+m),x,f_{m,n})$ the following are strongly collapsing levels:
 \begin{enumerate}
 \item  $k_{m,n}^{(3)},\ n\ne m-1$,
 \item $k_{m,n}^{(4)}$.
 \end{enumerate}
 \end{theorem}
 \begin{proof} We apply Theorem \ref{Criterion2}. By Theorem \ref{41}, we can choose  $S$ as in \eqref{S}.  
 Let $\C_t$ be the 1-dimensional representation of $\C\varpi$ with $\varpi$ acting by $\varpi\cdot 1=t$. Then, by \eqref{actionvarpi}, as a $(\C\varpi\oplus sl(n))$--module, 
\begin{equation}\label{SS} S=S_0\oplus (\C_{\frac{n+m}{m}}\otimes \C^n)\oplus (\C_{-\frac{n+m}{m}}\otimes (\C^n)^*),
\end{equation}
 where $S_0$ is the  isotypic component of the trivial  representation of $\g^\natural$. We now compute $C_1,\ldots,C_m$  (cf. \eqref{Cj}),  where $C_1,\ldots, C_{m-2}$ correspond
 to the $m-2$ copies of the trivial representation occurring in $S_0$ and  $C_{m-1}, C_m$ to the rightmost factors in \eqref{SS}. We have 
 $$C_1=\cdots=C_{m-2}=0,\quad C_{m-1}=C_m=(1-\d_{k_0,0})\frac{m+n}{2
   mnk_0}+\frac{n^2-1}{2n (k_1+n)}.$$ 
   A direct verification shows that if $k= k_{m,n} ^{(3)}$ then
 $$
 C_{m-1}=C_m=\frac{(m-1) \left(m+n^2+n-1\right)}{2 n^2},
 $$
 hence $C_{m-1}=C_m=\frac{m+1}{2}$ if and only if $n=m-1$.
 
  If $k=k_{m,n} ^{(4)}$ then  $C_{m-1}=C_m=\frac{m \left(n^2-1\right)}{2 n^2}$ hence, $C_{m-1}=C_m<\frac{m+1}{2}$.
 \end{proof}
 \section{Non-collapsing conformal  levels}
% { \color{blue} D: Unfortunately, we can prove here only that $k_{m,n}^ {(1)}$, $k_{m,n}^ {(2)}$ are not strongly collapsing.
% We still need to analyse possibility that some of the are collapsing.  
 %  I believe that $k_{m,n}^ {(1)}$ are not collapsing.
 %\item But surprisingly,  $k_{m,n}^ {(2)}$   is always  collapsing (see Corollary \ref{coll-2}).
%} 
 
% {\bf In this section we assume that $k=k_{m,n}^{(1)}$ and $\frac{m+1}{n-1} \notin {\Z}_{\ge 0}$, or   $k=k_{m,n}^{(2)}$ and $\frac{m}{n+1} \notin {\Z}_{\ge 0}$.}
Set 
$$I= W^k(\g,x,  f). (L - L^{\g^{\natural}}).$$
 \begin{theorem}\label{collnc} For the hook $W$-algebra $W(sl(n+m),x,f_{m,n})$ the  conformal levels $k_{m,n}^ {(1)}$ for $n >1$  and  $k_{m,n} ^{(2)}$    are not  strongly collapsing. \end{theorem}
%{\color{red} Here is an easy proof. Please check.}
\begin{proof} One readily computes that, if $k=k^{(i)}_{m,n}$, $i=1,2$, then $C_{m-1}=C_m=\frac{m+1}{2}$. We need only to check whether 
 $$  G_i ^{\pm} \notin   I.$$
 Set  $\bar L = W_2 = L - L^{\g^{\natural}}$.  Let $\tau \in \mbox{span}_{\C} \{ G_1^{\pm}, \dots, G_n ^{\pm}\}$.
Then for  each $i \in {\Z}_{\ge 1}$:
\begin{equation}\label{t}\tau_{i} \bar L =-[\bar L (-2), \tau_{i}] {\bf 1} = - \sum_{j=0} ^{\infty} {-1 \choose j} (\bar L(j-1) \tau )_{i-j-1} {\bf 1} =0. \end{equation}
 The last equality follows since, as shown in Theorem \ref{41}, the fields $G^\pm _i, i=1, \dots, n $  are primary for both $L$ and $V (\g^{\natural})$, hence $\bar L(j)\tau=0$ for $j\geq 0$.
For $2 \le p \le m$, we set $W_p (s) = (W_{p})_{s+p-1}$.
Then for $s \ge 1$ we have
$ W_p (s) \bar L = 0$. This is clear for $s\geq 3$; if $W_p (2) \bar L \ne 0$ then $I$ would  not be a proper ideal. Finally, if $s=1$ then $W_p (2) \bar L\in gl(n)_{-1}\vac$. If this vector is non-zero, then it belongs to 
$I\cap gl(n)_{-1}\vac$. This is not possible, since
\begin{align*}
\mathcal V(gl(n))&=V^{-\tfrac{m+n}{1+m}}(\C\varpi)\otimes V^{-\tfrac{1+mn}{1+m}}(sl(n))&&
\text{ if $k=k^{(1)}_{m,n}$,}\\
\mathcal V(gl(n))&=V^{1}(\C\varpi)\otimes V^{\tfrac{1+n(1-m)}{m}}(sl(n))&&\text{ if $k=k^{(2)}_{m,n}$,}\\
\end{align*}
and the levels appearing in the right hand sides are always non-zero (recall that we are assuming $n>1$).

Moreover $ W_p (0) \bar L = a \bar L + u$ for certain $a \in {\C}$ and $u \in V (\g^{\natural}) \cap I$.  Assume that $u$ is non-zero, then $u$ is a subsingular  vector in  $V (\g^{\natural})$ of zero $\g^{\natural}$-weight. Such  subsingular vector cannot exist since
$V (\g^{\natural})$ is not a vertex algebra at the critical level. Thus $u=0$ and we conclude that  $ W_p (0) \bar L = a \bar L$.

Therefore $\bar L$ is a singular vector in $ W^k(\g,x,  f)$. In particular, $I$ is a quotient of a Verma module, hence it is linearly spanned by monomials 
   $$ U_s =    y_1 (-q_1) \cdots  y_t (-q_t) x_1 (-n_1-1) \cdots x_r (-n_r-1) (g^1)_{-m_1} \cdots (g^s)_{-m_s} \bar L  , $$
 where $y_i \in \{ W_p, p=2, \dots, m\}$,   $x^i \in \g^{\natural}$,   $g^j \in \{ G_1 ^{\pm}, \dots, G_n ^{\pm}\}$,  $q_i, n_i  \in {\Z}_{\ge 0}$, and $m_i\in \Z,\, m_i\geq -\tfrac{m-1}{2},\,m_1\geq m_2\geq\ldots\geq m_s$. Formula
 \eqref{t} implies $m_s\geq 0$, hence all $m_i$ are non-negative integers.
 Denote the conformal weight  of $U_s$  by $\mbox{wt}(U_s)$ . Then we have

 $$ \mbox{wt} (U_s) \ge 2+ s\frac{m+1}{2} + (m_1 + \cdots + m_s) -s \ge 2 + s\frac{m-1}{2}. $$
 This implies that if $s \ge 1$, then $ \mbox{wt} (U_s) \ge \frac{m+3}{2}$.
 If $G_i ^{\pm} \in I$, then it can be written as a a linear combination of elements $U_{s_i}$ with $s_i \ge 1$ of conformal weight $\frac{m+1}{2}$. This is not possible. The claim follows.
\end{proof}

  We set \begin{equation}\label{pma}A^{\pm} = \mathcal V (\g^{\natural}). \mbox{span}_{\C} \{ G^{\pm}_i, i=1, \dots, n \},\end{equation}
and
 \begin{equation}\label{charge} \overline W_k(\g, x, f) = \bigoplus_{\ell \in {\Z}}   \overline W_k ^{(\ell)}, \quad 
   \overline W_k ^{(\ell )}= \{ v \in \overline W_k(\g, x, f) \ \vert \ J(0) v = \ell v \}, \end{equation}
{  where \begin{equation}\label{J}J=J^{\{-\tfrac{n+m}{m}\varpi\}}\end{equation}
   (cf. \eqref{varpi}).}
   
 { We believe (see Conjecture \ref{non-coll-1})   that  $k =k_{m,n}^{(1)}$ is never a collapsing level. We will prove this  in the special case when $k$ is admissible. The next result  supports  this  conjecture,  by showing 
   that each component $\overline W_k ^{(\ell )}$ is always  non-zero. In Remark  \ref{rem-non-coll-1} we shall explain how this result is related with Conjecture \ref{non-coll-1}.  }
  \begin{proposition}\label{75} Let $k =k_{m,n}^{(1)}$. For each $p \in {\Z}_{>0}$ we have
  $$ : (G_{i} ^{\pm} )^p : \ne 0  \quad \mbox{in}  \ \   \overline W_k(\g, x, f_{m,n}).   $$
  In particular, each   $\overline W_k ^{(\ell )}$ is a non-zero $\overline W_k ^{(0)}$--module.
  \end{proposition}
  \begin{proof}
 By direct calculation we have
\bea L^{\g^{\natural}} (0) :(G_{i} ^{\pm} )^p:  &=& \left(-\frac{(m+1) p ^2 }{2 n} + \frac{p (p+n) (m+1) }{2n} \right) :(G_{i} ^{\pm} )^p:  \nonumber \\
& =& \frac{p (m+1)}{2} :(G_{i} ^{\pm} )^p: .\nonumber 
\eea
This implies that for $s \ge 0$ we have
$$ \bar L(s) :(G_{i} ^{\pm} )^p: = 0. $$
%so  $ \bar M_m (0,0) = {\rm Vir}^{c=0}.:(G_{i} ^{\pm} )^p:$ is isomorphic to a certain quotient of the Verma module $M(0,0)$. 
Using 
%Proposition \ref{prop-fusion-trick} and
 the same arguments as in the proof of  Theorem \ref{collnc} we get $$:(G_{i} ^{\pm} )^p: \notin I =  W^k(\g,x, f_{m,n}) \bar L. $$
The claim follows.
  \end{proof}
 
 \begin{remark} {The situation for  $k =k_{m,n}^{(2)}$ is different. One can show that  for   $p \in {\Z}_{\ge 2}$ the following relation holds:
  $$ : (G_{i} ^{\pm} )^p :  = 0  \quad \mbox{in}  \ \   \overline W_{k_{m,n}^{(2)}}(\g, x, f_{m,n}).   $$
 This result won't be used elsewhere  in the paper.}
 \end{remark}

  \section{Decompositions of conformal embeddings}
  
Consider the embedding
  $$ \mathcal V(\g^{\natural}) =  V(\g^{\natural}). {\vac} \hookrightarrow  W_k(\g,x,  f). $$
  In the terminology of \cite{AKMPP} and \cite{AKMPP-JJM}, the conformal embedding $ \mathcal V(\g^{\natural})  \hookrightarrow W_k(\g, x, f)$ is called finite (resp. infinite) if each $W_k(\g,x, f)^{(i)}$ from \eqref{charge} is finite (resp. infinite) sum of $\mathcal V (\g^{\natural})$--modules.
    
 \subsection{Finite decompositions}  { Let $M(k)$ be the Heisenberg vertex algebras $V^k(\C J)$ (cf. \eqref{J}). Recall that 
 $V^{\beta_k}(\g^\natural)\subset W^k(\g,x,f)$ can be written as
 $$V^{\beta_k}(\g^\natural)=V^{k_1}(sl(n))\otimes M(k_0).$$
 Let $V_{sl(n)}(\mu)$ denote the finite-dimensional irreducible $sl(n)$-module of highest weight $\mu$. Let $L_k ^ {sl(n)} (\l)$ be the irreducible $V^k(sl(n))$-module with top component $V_{sl(n)}(\l)$. Identify $\C$  with  $(\C J)^*$ by 
$\l \mapsto \l(J)$,  and let $M(k,r)$ be the 
$M(k)$-module of highest weight $r$.
}
 We use the tensor product decomposition
\bea V_{sl(n)}(\omega_1) \otimes  V_{sl(n)}(\omega_{n-1} )
=  V_{sl(n)}(\omega_1 + \omega_{n-1}) \oplus  V_{sl(n)}(0).\label{tensor-1} \eea
Recall that $J(0)$ acts semi-simply on   $W_k(\g, x, f)$ and that 
 $$ W_k(\g, x, f_{m,n})= \bigoplus_{i \in {\Z}} W_k(\g, x, f_{m,n})^{(i)}, \quad  W_k(\g, x, f_{m,n})^{(i)}  = \{ v \in W_k(\g, x, f_{m,n}), \  J(0) v = i v  \}.$$

  \begin{lemma} \label{criterion}
 Assume that
 \item[(1)]   $\mathcal V(\g^{\natural})$ is conformally embedded in  $W_k(\g, x, f_{m,n})$ and the  level $k$ is not collapsing;
 \item[(2)] $W_k(\g, x, f)^{(0)}$ does not contain $\mathcal V (\g^{\natural})$--primitive vectors of  weight $\mu= \omega_1 + \omega_{n-1}$. 
 
 Then
  $W_k(\g,x,f_{m,n})^{(0)}$ is a simple vertex algebra isomorphic to $V(\g^{\natural})$  and each  \break $W_k(\g,x,  f_{m,n})^{(i)}$ is a  non-trivial simple $V(\g^{\natural})$--module.
  \end{lemma}
  \begin{proof}
  The proof is the same as the proof of  \cite[Theorem 6.2]{AKMPP-JJM}. 
    Since $\mathcal V(\g^{\natural})$ is conformally embedded into $W_k(\g, x,  f_{m,n})$ we have that $W_k(\g,x,  f_{m,n})$ is generated by $ \mathcal V(\g^{\natural})  + A^+ + A^-$ (see \eqref{pma} for notation).
 The assumption of the lemma and the tensor product decomposition (\ref{tensor-1}) give that
  \begin{equation}\label{fusion}
A^{+ } \cdot A^{-} \subset \mathcal V(\g^{\natural}).
  \end{equation}
  Since 
   $$ W_k(\g,x,  f_{m,n})^{(0)} \subset A_1 \cdots  A_r$$
with $A_i =A^{\pm}$ or $A= \mathcal V^{k^{\natural}}$ and
   $$\sharp \{ i   \ \vert \ A_i = A^+ \} = \sharp \{ i   \ \vert \ A_i = A^- \},
   $$
   by \eqref{fusion} we conclude that  $W_k(\g, x, f_{m,n})^{(0)} =\mathcal V(\g^{\natural})$. Since $W_k(\g, x, f_{m,n})^{(0)}$ is a simple vertex algebra, we have that $\mathcal V(\g^{\natural}) =   V(\g^{\natural})$ is simple, and each 
    $W_k(\g,x,f_{m,n})^{(i)}$ is a simple $V(\g^{\natural})$--module.
  \end{proof}
  
 %   In this section we decompose conformal embeddings in the non-collapsing cases $k  =k_{m,n} ^{(1)}$ and $k= k_{m,n} ^{(2)}$.
  
%Consider first case $k=k_{m,n} ^{(1)}$. We can  get the decomposition provided that $\frac{m+1}{n-1} \notin {\Z}$.
% {\color{blue} and that $G^{\pm} _i$ are non-zero.}

Using the tensor product decomposition \eqref{tensor-1}, we get that any $V(\g ^{\natural})$--primitive vector of $sl(n)$--weight $\mu= \omega_1 + \omega_{n-1}$ has  conformal weight
\begin{itemize}
\item $h_{\mu} ^{(1)} = m+1 + \frac{m+1}{n-1}$ if $  k  =k_{m,n} ^{(1)}$,
\item $h_{\mu} ^{(2)} = m - \frac{m}{n+1}$ if $k  =k_{m,n} ^{(2)}$.
\end{itemize}

\begin{theorem} \label{decomp}     
% ( {\color{blue} Assuming that  Lemma \ref{non-zero-2} holds}).
 Let $k=k_{m,n}^{(i)}$ for  $i=1,2$  and assume  that     $k$ is non-collapsing.   Assume also that $\frac{m+1}{n-1} \notin {\Z}$  if  $i=1$ and $ \frac{m}{n+1}  \notin {\Z}$ if $i=2$. Then
 
\bea \label{deccc} W_k = W_k(\g, x,  f_{m,n}) = \bigoplus_{i \in {\Z}}  W_k ^{(i)}, \label{dec-nat} \eea
and each $W_k ^{(i)} = \{ v \in  W_k \ \vert \ J(0) v = i v \}$ is an irreducible $V(\g^{\natural})$--module. { 
The summands in the r.h.s. of  	\eqref{deccc} have the form:
\begin{itemize}
\item $ W_k ^{(i)}= L_{k_1} ^{sl(n)} ( i \omega_1) \otimes M(k_0, i)$ if $i \ge 0$,  
\item $ W_k ^{(i)}= L_{k_1} ^{sl(n)} ( -i \omega_{n-1}) \otimes M (k_0, i)$ if $i <0$.
\end{itemize}
In particular, $\mathcal V(\g^{\natural}) \cong W_k(\g, x, f_{m,n})^{(0)} = V(sl(n)) \otimes V^{k_0}(\C J)$ is  a simple vertex algebra  which is conformally embedded in $W_k(\g, x, f_{m,n})$.
}
\end{theorem}
\begin{proof}
 Assume that $k$ is non-collapsing. Then by the simplicity of   $W_k$   we conclude that $ A^+ \cdot A^-  \ne \{ 0\}$ in $W_k$. 
This implies    that $W_k^{(0)}$ has a primitive vector $v_{\nu}$  of $\g^{\natural}$--weight $\nu \in  \{ 0, \mu\}$.   Since  $h^{(i)}_{\mu} \notin {\Z}$, we conclude that $W_k ^{(0)}$ can not contain primitive vector $v_{\mu}$.
  Now the  claim follows from Lemma \ref{criterion}. %This proves the  assertion.

\end{proof}

\begin{remark} \label{rem-non-coll-1}In Theorem  \ref{non-collapsing-admissible} we shall prove that  $k^{(i)}_{m,n},\,i=1,2,$ are non-collapsing whenever they are  admissible.  But we expect that the level  $k=k_{m,n}^{(1)}$ is also  non-collapsing for non-admissible level. Since we cannot prove this, let us contemplate  what could happen if $k$ is collapsing and $\frac{m+1}{n-1} \notin {\Z}$.  Then:
\begin{itemize}
\item  
$:(G _i ^{\pm}) ^p:  \ne 0 \quad \mbox{in} \  \overline W_k(\g, x,f_{m,n})$, for each $p \in {\Z}_{\ge 0}$ (cf. Proposition \ref{75});
\item $A^+ \cdot A^- =\{0\}$.
\end{itemize}
This gives that
$  \Omega_{ \pm p} =:(G _i ^{\pm}) ^p:$
is a non-zero singular vector in $\overline W_k(\g, x,f_{m,n})$. In particular, 
%\begin{itemize}
\begin{itemize}
\item $\overline W_k(\g, x,f_{m,n})$ has infinite-length as $W^k(\g, x,f_{m,n})$--module.
%\end{itemize}
\item[]Since $ W^k(\g, x,f_{m,n}) =H_f (V^k(\g))$, we conclude that:
%\begin{itemize}
\item $V^k(\g)$ is a module in $KL^k$ of infinite length.
\end{itemize}
 So assuming  that $k=k_{m,n} ^{(1)}$ is collapsing would imply the existence of highest weight modules in $KL^k$ of infinite length. It is expected that this is not possible for non-critical levels. On  the other hand, if $k=k_{m,n}^{(2)}$ is non-admissible, we will see in Proposition  \ref{collapsing-3p2}  that $k$ can be collapsing.
\end{remark}

\begin{remark}\label{AAAA}   Note that the levels   $k^{\natural} = -n + \frac{n-1}{m+1} = -n + \frac{n-1}{q}$, for   $q=m+1$, appeared in 
 \cite[Remark 8.9]{AEM}. The authors  conjectured that these levels are collapsing  if $n= p q$ and that 
 $W_{k^{\natural}} (sl(n), f_1) = V_{-1} (sl(p))$ for certain nilpotent element $f_1$. In Section \ref{rectanguar} we prove that their conjecture is true except possibly for $p=2$.
\end{remark}

\subsection{Infinite decompositions: the case $n=2$}
 We shall see that if the conformal level does not satisfy the condition of Theorem \ref{decomp}, then   we can expect that the decomposition is infinite.  
 
 Let us consider here   the case $m = p-1$, $n =2$.  Then $k=k_{p-1,2} ^{(1)}$  and $k_1= -2 + \frac{1}{p}$. Clearly, the conditions of Theorem \ref{decomp} are not satisfied. But according to Theorem  \ref{MT} we have conformal embedding
 $$ V(\g^{\natural}) = V_{-2+ \frac{1}{p}} ( sl(2)) \otimes M(k_0)
 \hookrightarrow W_k(\g, x, f_{p-1,2}).$$

  In the analysis of this conformal embedding, an  important role is played by  the vertex algebras $\mathcal V^{(p)}$ and $\mathcal R^{(p)}$ introduced in \cite{A-TG}.
Their properties are further studied in \cite{ACGY}.
  The vertex algebra  $\mathcal V^{(p)}$ is realized as an extension of $V_{-2+ \frac{1}{p}} (sl(2))$. It has ${\Z}$--gradation so that
  \begin{align*}\mathcal V^{(p)} &= \bigoplus_{\ell \in {\Z}} \mathcal V^{(p)} _{\ell},\\
   \mathcal V^{(p)} _{\ell} &= \bigoplus_{s=0} ^{\infty} L^{sl(2)} _{-2+1/p} ( (\vert \ell \vert + 2 s) \omega_1).\end{align*}
   $\mathcal V^{(p)}$ is generated by the generators $e,f,h$ of  $V_{-2+1/p}(sl(2))$ and four primary fields $\tau^{\pm}, \overline{\tau} ^{\pm}$ of conformal weight $\frac{3p}{4}$. In the case $p=2$, $\mathcal V^{(p)}$ is isomorphic to the small  $N=4$ superconformal vertex algebra with central charge $c=-9$.
  
  For each $q \in \tfrac{1}{2} {\Z}$, let $F_q$ denote the (generalized) lattice vertex algebra $V_{ \Z \varphi} =M_{\varphi} (1) \otimes {\C}[\Z \varphi]$ such that $\langle \varphi, \varphi \rangle = q$. Here $M_{\varphi} (1)$ is the Heisenberg vertex algebra  of level $q$ generated by the field $\varphi(z) = \sum_{n \in {\Z}} \varphi(n) z^{-n-1}$  and ${\C}[\Z \varphi]$ is the group algebra of the lattice ${\Z} \varphi$. Let $M_{\varphi} (1, r)$ denotes the irreducible $M_{\varphi} (1)$--module on which  $\varphi(0)$ acts as $r \cdot \mbox{Id}$. Then
  $$ F_q = \bigoplus _{\ell \in {\Z}} F_{q}  ^{(\ell)}, \quad  F_{q} ^{(\ell)} \cong M_{\varphi} (1, q\ell).$$
  Finally, recall that the vertex algebra $\mathcal R^{(p)}$ is realized as the following subalgebra of $\mathcal V^{(p)} \otimes F_{-\frac{p}{2}}$:
  $$ \mathcal R^{(p)} = \bigoplus _{\ell \in {\Z}}\mathcal  V^{(p)} _{\ell} \otimes F_{-\tfrac{p}{2}} ^{(\ell)} = \bigoplus _{\ell \in {\Z}} \mathcal V^{(p)} _{\ell} \otimes M_{\varphi}(1, -\ell \tfrac{p}{2}). $$
    The vertex algebra $\mathcal R^{(p)}$ admits the decomposition
  $$ \mathcal R^{(p)} = \bigoplus  \mathcal R^{(p)}_{(\ell)}, $$
  such that each  $\mathcal R^{(p)}_{(\ell)}= \mathcal V^{(p)} _{\ell} \otimes M_{\varphi}(1, -\ell \tfrac{p}{2})$ is a direct sum of infinitely many non-isomorphic  $V_{-2+ \frac{1}{p}} (sl(2)) \otimes M_{\varphi} (1)$--modules.
  The next theorem is proved in \cite[Theorem 10]{ACGY}:
  %  by using the tensor category theory in $KL_{-2+1/p}$.
  
  \begin{theorem} \label{rp} \cite{ACGY} Let $k=k_{p-1,2} ^{(1)} = -\frac{p^2-1}{p}$. Then
 $W_k(sl(p+1), f_{p-1,2}) =  \mathcal R^{(p)}$. 
In particular, level $k$ is not collapsing.
\end{theorem} 
  As a consequence, the decomposition of conformal embedding is infinite in this case.
   Some important special cases were considered in our previous papers:
  \begin{itemize}
  \item the case $p=2$ was treated in \cite{A-TG}, when it was proved that $W_k(\g,x, f)  =V_{-3/2}( sl(3))  = \mathcal R^{(2)}$. This result is used in \cite{AKMPP} to show that the decomposition of the conformal embedding $gl(2) \hookrightarrow sl(3)$ is infinite for $k=-3/2$;
  \item the case $p=3$ was treated in \cite{AKMPP-JJM}, where it was proved that $W_k(\g,x,f) = W_k(\g,x, f_\theta) = \mathcal R^{(3)}$.
\end{itemize}

\subsection{ More collapsing levels} Let $k=k_{m,n} ^{(2)}$ for $m = 3p$ and $n=2$. We see that again  the conditions of Theorem \ref{decomp} are not satisfied. Then  we have:
\begin{itemize}
\item $k_1 = -2 + \frac{1}{p}$.
\item The embedding
$$ V (\g^{\natural}) = V_{-2+ \frac{1}{p}} ( sl(2)) \otimes M(k_0)
 \hookrightarrow \overline W_k(\g,x,f_{3p,2})$$
 is conformal.
 %\item But $k=k_{3p,2} ^{(2)}$  is also collapsing level because of Corollary \ref{coll-2}.
 %\item $W_k(\g, f)$ is a simple vertex algebra generated by $V^{k^{\natural}}(\g^{\natural})$ and  primary fields $G^{\pm} _i$, $i=1,2$ of conformal weight $\frac{3p+1}{2}$.

%\item {\color{blue}D:  I believe that $k=k_{3p,2} ^{(2)}$ is an example of collapsing level which is not strongly collapsing.}
\end{itemize}
%\begin{comment}
\begin{proposition} \label{collapsing-3p2} Level $k= k_{3p,2}^{(2)} = -3p-1 + \frac{1}{p}$ is collapsing and
$$W_k(sl(3p+2), x, f_{3p,2}) = V_{-2+\frac{1}{p}} (sl(2)) \otimes M(k_0). $$
\end{proposition}
\begin{proof}
 Set $W_k =W_k(sl(3p+2), x, f_{3p,2})$. Assume that $k$ is not collapsing. Then $W_k$ is strongly generated by  generators of $V_{-2+ \frac{1}{p}} ( sl(2)) \otimes M(k_0)$ and four primary fields $G^{\pm} _i$, $i =1,2$,    of conformal weight $(3p+1)/2$. 

Assume next that $W_k ^{(0)}$ contains a primitive vector $u$ of $\g^{\natural}$--weight $\mu=\omega_1+\omega_{n-1}$. Then the conformal weight of $u$ is $h^{(2)} _{\mu} = 2p $.
Since $u$ is strongly generated by affine generators and $G^{\pm} _{i}$, we have that
$$ u = \sum_{i, j=1} ^2 a_{i,j} (G^+ _i )_{-n_0} (G^- _j ) + u', $$
where $n_0 >0$,  $a_{i,j} \in {\C}$ (not all zero) and $u' \in V_{-2+ \frac{1}{p}} ( sl(2)) \otimes M(k_0)$. This implies that the conformal weight $h_{\mu} ^{(2)} \ge 3p +1$. This is a contradiction.  

Therefore $W_k ^{(0)}$ cannot contain a primitive vector of $\g^{\natural}$--weight $\mu$. Lemma \ref{criterion} implies that $W_k ^{(0)} = V_{k_1} ( sl(2)) \otimes M(k_0)$, where $k_1 = -2+ 1/p$,  and 
$$ W_k ^{(i)} = L^{sl(2)} _{k_1} (\vert i\vert  \omega_1) \otimes  M(k_0, i).$$
Therefore
\bea \label{dec-3p2}W_k \cong \bigoplus_{i \in {\Z}}   L^{sl(2)} _{k_1} (\vert i\vert  \omega_1) \otimes  M(k_0, i). \eea

Using arguments from \cite[Section 5]{CY} we can conclude that, since $KL ^{sl(2)} _{k_1}$ is a braided tensor category,  then the $V_{k_1}(sl(2))$--modules and  the $M( k_0 )$--modules appearing in the decomposition (\ref{dec-3p2}) must have same fusion rules. This implies  that  all modules $L^{sl(2)} _{k_1}(i  \omega_1)$ are simple currents in $KL ^{sl(2)} _{k_1}$. This is a contradiction since   the fusion  ring of   $KL ^{sl(2)} _{k_1}$ is equivalent to
Grothendieck ring of finite dimensional $sl(2)$--modules (cf. \cite{ACGY}).  Therefore the modules $L^{sl(2)} _{k_1}( i   \omega_1)$ are not simple currents for $i >0$.
 We have therefore proved that $k$ is a collapsing level. \end{proof}

\section{ The cases when the conformal level is admissible}
 Recall the following definition. 
\begin{defi}
A level $k$ for $\g=sl(n+m)$ is  said to be admissible if $k+ h^{\vee} = \frac{p'}{p}$, $p,p' \in {\Z}_{\ge 1}$, $(p,p')=1$ and $p' \ge h^{\vee}=n+m$.
\end{defi}
In this section we study the cases when $k=k_{m,n}^{(1)}$ or $k =k_{m,n}^{(2)}$ are  admissible rational numbers.

We  can choose the root system so that the  $sl(2)$--triple corresponding to highest root $\theta$ is
\bea
&&e_{\theta}=\left(\begin{array}{c|c}
  (J_{m} ^{tr} )  ^{m-1}& 0 \\
\hline
 0 & 0 
\end{array}\right),\quad h_{\theta}=\left(\begin{array}{cccc|c}
  1&0&\cdots&0 & \\
0& 0&\cdots&0 &  \\
\vdots&\vdots&\ddots&\vdots&0\\
 0&0&\cdots &-1 &  \\
 \hline
 &&0&&0
\end{array}\right), \quad  f_{\theta}=\left(\begin{array}{c|c}
  J_{m}  ^{m-1}& 0 \\
\hline
 0 & 0 
\end{array}\right). \label{h-root}
\eea

 The following  results from \cite{KW88} and \cite{AE21} hold  for any simple Lie algebra $\g$.
\begin{proposition} \cite{KW88} Assume that $k$ is an admissible level. Then 
$$ V_k(\g) = V^k(\g)\big/J^k(\g),$$
where the maximal ideal $J^k(\g) = V^k(\g). \Omega _{k}$ is generated by  the singular vector $\Omega _{k}$, which  is the  unique (up to a scalar factor) singular vector in  $ V^k(\g)$ of $\g$--highest weight $ \mu^{(k)}  = (p' +1 - h^{\vee}) \theta$, and conformal weight $ d_{(k)}= p (p '+1-h^{\vee})$.
\end{proposition}

{

Denote by $H_f$ the quantum Hamiltonian reduction functor. In \cite{KW08} Kac and Wakimoto state the following conjecture:
\begin{equation} \label{KWW}\text{
$ H_f ( V_k(\g)) $ is either zero or isomorphic to $W_k(\g, x,f)$.}
\end{equation}
 Recall from \cite{AARA} that if $k=-h^\vee+p'/p$ is admissible, then the associated variety of the simple affine vertex algebra $V_k(\g)$ is the closure of a  nilpotent orbit $\mathbb O_{k}$, depending just on $p$.  Moreover, if we set 
 $N_p=\{ x \in \g \ \vert \ \mbox{ad} (x) ^{2p} = 0\}$, then 
$$ \overline{\mathbb O}_{k} = N_p.$$
The following result  has been proved in  \cite[Theorem 7.8]{AE21}.
\begin{theorem} (Arakawa-Van Ekeren) \label{ave} Assume that  $k$ is an admissible level,  $f\in \overline{\mathbb O}_{k}$ and $f$ admits an even good grading. Then \eqref{KWW} holds.
\end{theorem}
}
The following result is a consequence of results from \cite{Ar11} (see also \cite[Remark 7.10]{AE21}). 
 \begin{proposition}   \label{non-triv}  Assume that  \eqref{KWW} holds for $k$ admissible. Then
$H_f ( J^k(\g)) $ is a  submodule of $W^k(\g, x, f) = H_f ( V^k(\g))$ which is generated by a   singular vector
$\Omega_k ^{W}$ of conformal weight
$$ d_k ^{W} =  (p'+ 1 - h^{\vee}) (p- (h_\theta | x)). $$
%$$ d_k ^{W} = \frac{ (\mu^{(k)}, \mu ^{(k)} + 2 \rho) }{k + h^{\vee} }- \frac{k+h^{\vee}}{2}  \vert x \vert ^2  + (x, \rho). $$
Moreover,   $H_f ( J^k(\g))$ is either a maximal ideal in  $W^k(\g, x,f)$ or $ H_f ( J^k(\g)) = W^k(\g, x,f)$.
If $f \in  \overline{\mathbb O}_{k}$ and $f$ admits an even good grading, then the above statement holds.
% =\overline{\mathbb O}_{p} =\{ x \in \g \ \vert \ \mbox{ad} (x) ^{2p} = 0\} $.
\end{proposition}
\begin{proof}
%{\color{blue} Need to be checked}
It is proved in \cite{Ar11} that $H_f$ is an exact functor acting on $KL_k$. This implies that  that
$$ H_f( V_k(\g)) =  H_f ( V^k(\g))\big/H_f ( J^k(\g)) = W^k(\g, x,f)\big /H_f ( J^k(\g)).$$ 
By Conjecture  \ref{KWW}
$H_f ( J^k(\g))$ is either a maximal ideal or  $ H_f ( J^k(\g)) = W^k(\g, x,f)$. In any case, we have that $ H_f ( J^k(\g))$ is a non-zero submodule in $W^k(\g, x, f)$. Then the lowest conformal weight of  $H_f ( J^k(\g))$ is  given by the formula given in \cite[Remark 2.3]{KRW}, 
%$$ d_k ^{W} = \frac{ (\mu^{(k)}, \mu ^{(k)} + 2 \rho) }{k + h^{\vee} }- \frac{k+h^{\vee}}{2}  \Vert x \Vert ^2  + (x, \rho), $$
 which can be expressed as $ d_k ^{W} = (p'+ 1 - h^{\vee}) (p- (h_\theta | x))$. { The final claim follows from Theorem 
 \ref{ave}}. 
\end{proof}

%Recall that the choice of the set of positive roots in $\g$ must be compatible  with the gradation given by $x$ (i.e., $\bigoplus\limits_{j>0}\g_j\subset \bigoplus\limits_{\a>0}\g_\a$). If $\theta$ is the highest root of such a positive system, 
By \eqref{xf}, we have 
$(h_\theta|x) = m-1$. Since $h^{\vee} = n+m$ ,  we get
\begin{lemma}\label{94} We have:
\begin{itemize}
\item[(1)] $k=k_{m,n} ^{(1)}$ is admissible if and only if $(n-1,m+1) =1$.  
\item[(2)] $k=k_{m,n} ^{(2)}$ is admissible if and only if $(n+1, m) =1$.  
%\item[(3)] $k_{m,n} ^{(3)}$ is never admissible.
%\item[(4)] $k_{m,n} ^{(4)}$ is   admissible if and only if $(m,n) =1$.
\end{itemize}
In both cases,  if $k$ is admissible we have $ d_{ k} ^{W} = 2$. 
\end{lemma}

\begin{lemma}\label{95} {We have:
$$f=f_{m,n} \in N_{m}\subset  N_{m+1}.$$}
\end{lemma}
\begin{proof}
Recall that the height of a nilpotent element $f$ is 
$\max\{n\mid ad(f)^n\ne0\}$. If $\{e,h,f\}$ is a $sl(2)$--triple containing $f$, and one chooses the set of positive roots so that $h$ is dominant, then the height of $f$ is $\theta(h)$, with $\theta$ the highest root. The dominant element corresponding to $f_{m,n}$ is $h=2x$ and the Dynkin labels of $x$ are given in \eqref{Dyneven} and \eqref{Dynodd}. It follows that the height of $f_{m,n}$ is $2(m-1)$.
\end{proof}
%\begin{remark}
%{ \color{blue} We need that 

%\begin{itemize}
%\item  $f=f_{m,n} \in \overline{\mathbb O}_{m+1}$, which is equivalent to
%$$ ad(f) ^{2m+2} =0. $$
%\item   $f=f_{m,n} \in \overline{\mathbb O}_{m}$, which is equivalent to
%$$ ad(f) ^{2m} =0. $$
%\item $\overline{\mathbb O}_{m} \subset  \overline{\mathbb O}_{m+1}$.
%\end{itemize}
%Is this correct?
%}
%\end{remark}

Th next theorem completely describes the structure of $W_k(\g, x,f)$ in the admissible case:
\begin{theorem} \label{non-collapsing-admissible} Assume that $k=k_{m,n} ^{(1)}$ or $k=k_{m,n} ^{(2)}$ and that $k$ is admissible. We have:
\begin{itemize}
\item[(1)] $W_k(\g, x,f_{m,n}) = \overline W_k(\g, x, f_{m,n})$.
\item[(2)]   The level $k$ is not collapsing.
\item[(3)] $W_k(\g, x,f_{m,n})$ admits the decomposition given in  Theorem  \ref{decomp}  provided that $k\ne k_{p-1,2} ^{(1)}$.
\item[(4)] If $k= k_{p-1,2} ^{(1)}$, the decomposition is given in Theorem \ref{rp}.
\end{itemize}
\end{theorem}
 
 \begin{proof}
 % {\color{blue} Assuming that $f_{m,n} \in \overline{\mathbb O}_{m+1} \bigcap  \overline{\mathbb O}_{m} $.}
 
Let $k=k_{m,n} ^{(1)}$  { be admissible. Then, by Lemma  \ref{94},  $\overline{\mathbb O}_{k}=N_{m+1}$, hence, by Lemma  \ref{95},
$f_{m,n}\in \overline{\mathbb O}_{k}$. In Remark \ref{good} we exhibited a good grading for $f_{m.n}$. Therefore, by Proposition \ref{non-triv} we have that  $H_f(V_k(\g))$ is zero or $W_k(\g,x,f_{m,n})$. By Lemma \ref{94} the former possibility is excluded, since 
$d^W_k=2$. In particular, it follows   that } the maximal ideal in $W^k(\g, x, f_{m,n})$ is generated by a singular vector $\Omega^W_k$ of conformal weight $2$. Since the highest  weight of $\Omega_k$ with respect to $\g$ is $(p+1-h^{\vee}) \theta$, using (\ref{h-root})  one shows that the  singular vector $\Omega_k^{W}$ must have $\g^{\natural}$--weight $0$. Since $k_1$ is not critical, the only possible candidate for a singular vector of conformal weight $2$ of $\g^{\natural}$--weight $0$ is $\bar L$. Therefore $\Omega_k ^{W} = \bar L$. This implies that the maximal ideal is $W^k(\g, x, f_{m,n}) $ is
$ W^{k}(\g, x,f_{m,n}) . \bar L$, and thus 
  $W_k(\g, x,f_{m,n}) = \overline W_k(\g, x, f_{m,n})$, proving (1). Therefore, {by Theorem \ref{collnc},} $k$ is non-collapsing. Since $n >2$  and  $(n-1,m+1) =1$, we conclude that $\frac{m+1}{n-1} \notin {\Z}$. Therefore the assumptions of Theorem \ref{decomp} are satisfied and we get assertion (3).
The case $k=k_{m,n} ^{(2)}$ is dealt with along the same lines.
 
 \end{proof}

  \section{The rectangular case and the Arakawa-Van Ekeren-Moreau question}\label{rectanguar}

 Let $f=f_{[m,q]}=diag(\underbrace{J_q,\ldots,J_q}_{\text{$m$ times}})$ be the nilpotent element of $sl(qm)$ with Jordan block decomposition corresponding to the partition $(q^m)$. Let $x=diag(\underbrace{\tfrac{q-1}{2},\tfrac{q-3}{2},\ldots,\tfrac{1-1}{2}}_{\text{$m$ times}})$ be the corresponding Dynkin element, whose weighted Dynkin diagram is 
$$ (\underbrace{\underbrace{0,\ldots,0}_{m-1},1,\underbrace{0,\ldots,0}_{m-1},1,\ldots\ldots,\underbrace{0,\ldots,0}_{m-1},1}_{q-1},\underbrace{0,\ldots,0}_{m-1}).$$
Using formula \eqref{calcolodim} to obtain Table 3, we get 
 \begin{equation}\label{gf}\g^f=\left(\begin{array}{c|c|c}
A_{11}&\ldots&A_{1m}\\
\hline
\vdots&\cdots&\vdots\\
\hline
A_{m1}&\ldots&A_{mm}
 \end{array}\right)\end{equation}
 where $$A_{ij}=\a^{ij}_0 Id+\a^{ij}_1J_q+\ldots+\a^{ij}_{q-1}J_q^{q-1}.$$
 Indeed, the matrices in the r.h.s. of \eqref{gf} are contained in $\g^f$; comparing dimensions we have the equality. Moreover
\begin{align}\label{r1}\g^\natural&=\left(\begin{array}{c|c|c}
\a^{11}_0Id&\ldots&\a^{1m}_0Id\\
\hline
\vdots&\ldots&\vdots\\
\hline
\a^{m1}_0Id&\ldots&-\sum_{j=1}^{m-1}\a^{jj}_0Id
 \end{array}\right)\cong sl(m),\\\notag
  \\\label{r2}
\g^f_{-r}&=\left(\begin{array}{c|c|c}
\a^{11}_rJ_q^r&\ldots&\a^{1m}_rJ_q^r\\
\hline
\vdots&\ldots&\vdots\\
\hline
\a^{m1}_rJ_q^r&\ldots&\a^{mm}_rJ_q^r
 \end{array}\right)\cong gl(m), \ \ 1\leq r\leq q-1.\end{align}

 \begin{table} 
\vskip 5pt
\begin{tabular}{c|c | c}
 & $\dim\mathfrak g_{-j}$ &$\dim\mathfrak g_{-j}^f$\\
\hline
$1\leq j\leq q-1$&$m^2(q-j)$&$m^2$\\\hline
$j=0$&$m^2q-1$&$m^2-1$\\
\end{tabular}\vskip 5pt\vskip 5pt\vskip 5pt
\caption{}
\end{table}

Applying Theorem \ref{structure} we obtain
 \begin{theorem}\label{genrect}% Assume that $k$ is not critical, i.e. $k \ne -n-m $. 
One can choose strong generators for the vertex algebra  $W^k(\g,x, f_{[m,q]})$ as follows: 
\begin{itemize}
\item[(1)]  $J^{\{a\}},\,a\in\g^{\natural}\cong sl(m)$; these generators are primary  for $L$ of conformal weight $1$;
\item[(2)]the Virasoro field $L$;
\item[(3)]   fields $W_i$, $3\leq i\leq  q,$ of conformal weight $i$;
\item[(4)]   fields $G^{j,s}_i$, $2\leq i\leq  q,\ 1\leq s,j \leq m^2-1$, of conformal weight $i$.
\end{itemize}
The action of $\g^\natural$ on $W_i$ (resp. $G^{j,s}_i$) is trivial (resp. adjoint).
\end{theorem}

%Then 
 %\begin{equation}\label{decrect}\g^\natural \cong sl(m),\quad\g^f_i\cong gl(m)\text{ for $i=1,\ldots,q-1$.}\end{equation}

According to \cite{KW}, the central charge for $L$ in $W^k(sl(mq),f_{[m,q]},x)$ is 
$$
C(k)=\frac{k \left(m \left(q-q^3\right)
   (k+m q)+m^2 q^2-1\right)}{k+m
   q}-m^2 q \left(q^3-2 q^2+1\right).
$$
The hypothesis of Theorem \ref{Criterion} are satisfied since we can choose $S\cap W^k(\g,f_{[m,q]},x)(2)$ to be, with the notation of Theorem \ref{structure}, $\Psi(sl(m))$.
Thus the embedding $\mathcal V(\g^{\natural}) \hookrightarrow W_k(\g, x,f_{[m,q]})$ is conformal if and only if $k$ is a solution of the equation $C(k)=c_{\g^\natural}$.
One readily computes that 
$$
c_{\g^\natural}=\frac{\left(m^2-1\right) q (k+m
   q-m)}{q (k+m q-m)+m},
$$
hence the conformal levels are
\begin{equation}\label{levels}
k^{[1]} _{m,q}=-\frac{m q^2}{q+1},\ k^{[2]}_{m,q}=\frac{-m q^2+m q-1}{q}
,\ k^{[3]} _{m,q}=\frac{-m q^2+m q+1}{q}.
\end{equation}

\begin{theorem} \label{collapsing-rectangular}
Levels $k^{[i]}_{m,q}$, $i=1,2,3$ are collapsing for all $q\ge 2$ and $m\ge 2$,  except possibly level $k^{[2]}_{2,q}$.
More precisely:
\bea 
W_{k^{[1]} _{m,q} }(sl(mq),x,f_{[m,q]})&=&V_{-\frac{m q}{q+1}}(sl(m)),  \label{coll-k1}\\
W_{k^{[2]} _{m,q} }(sl(mq),x,f_{[m,q]})&=&V_{-1}(sl(m)),\ m\ge 3,  \label{coll-k2}   \\
W_{k^{[3]} _{m,q} }(sl(mq),x,f_{[m,q]})&=&V_{1}(sl(m)). \label{coll-k3} 
\eea
\end{theorem}
\begin{proof} We use  Theorem \ref{Criterion2}. Recall that $k_1=qk+mq^2-mq$.  Also remark that, by \eqref{r2}, $\g^f$ decomposes as a sum of trivial 
and adjoint representations. Since $\g^\natural$ is simple, we have  (cf. \eqref{Cj})
$$C_i=\begin{cases}0\quad&\text{if $S_i$ is trivial,}\\\frac{2m}{2(k_1+m)}\quad&\text{if $S_i$ is adjoint.}\end{cases}$$ 
Assume $C_i\ne 0$. If $k=k^{[1]}_{m,q} $, then $C_i=q+1$, which cannot equate  $\D_i$ (see Theorem \ref{Criterion2}), since the maximal conformal weight of a generator  is $q$.
If $k=k^{[2]}_{m,q}$, then $C_i=\tfrac{m}{m-1}$, which might be integral iff $m=2$. If $k=k^{[3]}_{m,q} $, then $C_i=\tfrac{m}{m+1}$, which is never  integral.
\end{proof}
 \begin{remark} Note that 
 \begin{itemize}
 \item 
  $k^{[1]} _{m,q}$ is admissible if and only if $(m, q+1) = 1$. 
  \item $k^{[2]} _{m,q}$ is never admissible. The result (\ref{coll-k2})  gives a positive answer to a  question by T. Arakawa, J. van Ekeren and  A. Moreau  \cite[Remark 8.9]{AEM} in  the case $m \ge 3$.
  \item $k^{[3]} _{m,q}$ is always admissible. The result (\ref{coll-k3}) is proved by different methods in \cite{AEM}. 
 \end{itemize}
  
 \end{remark}

 Theorem \ref{collapsing-rectangular} cannot be applied for $m=2$. 
In the case $m=q=2$,  $f_{[m,q]}$ coincides with  the short nilpotent element $f_{sh}$. It is proved in \cite{AMP-21},   using explicit OPE formulas, that
$W_{-\tfrac{5}{2}} (sl(4), f_{sh})$ is isomorphic to an orbifold of the rank two Weyl vertex algebra. So $k^{[2]}_{2,2} = -5/2$ is not collapsing.

\begin{conj} Level $k = k^{[2]}_{2,q}=\frac{-2 q^2+2 q-1}{q}$ is not collapsing  for $q \ge 2$  and
$$W_{k} (sl(2q),f_{[2,q]}, x) = W_{-\tfrac{5}{2}} (sl(4), f_{sh}). $$
\end{conj}

  \vskip5pt
   \footnotesize{
  \noindent{\bf D.A.}:  Department of Mathematics, Faculty of Science, University of Zagreb, Bijeni\v{c}ka 30, 10 000 Zagreb, Croatia;
{\tt adamovic@math.hr}

\noindent{\bf P.MF.}: Politecnico di Milano, Polo regionale di Como,
Via Anzani 42, 22100 Como,\newline
Italy; {\tt pierluigi.moseneder@polimi.it}

\noindent{\bf P.P.}: Dipartimento di Matematica, Sapienza Universit\`a di Roma, P.le A. Moro 2,
00185, Roma, Italy;\newline {\tt papi@mat.uniroma1.it}
}
 \end{document}